\documentclass[12pt,reqno]{amsart}
\usepackage{amssymb}
\usepackage{amscd}
\usepackage{hyperref}

\newtheorem{theorem}{Theorem}[section]
\newtheorem{corollary}[theorem]{Corollary}
\newtheorem{lemma}[theorem]{Lemma}
\newtheorem{proposition}[theorem]{Proposition}

\newtheorem{definition}[theorem]{Definition}
\newtheorem{remark}[theorem]{Remark}

\newtheorem{exercise}{Exercise}

\numberwithin{equation}{section}
\begin{document}
\title[Transversal Dirac Operators]{Transversal Dirac operators on
distributions, foliations, and $G$-manifolds\\
Lecture Notes}
\author[K.~Richardson]{Ken Richardson}
\email[K.~Richardson]{k.richardson@tcu.edu}

\begin{abstract}
In these lectures, we investigate generalizations of the ordinary Dirac
operator to manifolds with additional structure. In particular, if the
manifold comes equipped with a distribution and an associated Clifford
algebra action on a bundle over the manifold, one may define a transversal
Dirac operator associated to this structure. We investigate the geometric
and analytic properties of these operators, and we apply the analysis to the
settings of Riemannian foliations and of manifolds endowed with Lie group
actions. Among other results, we show that although a bundle-like metric on
the manifold is needed to define the basic Dirac operator on a Riemannian
foliation, its spectrum depends only on the Riemannian foliation structure.
Using these ideas, we produce a type of basic cohomology that satisfies
Poincar\'{e} duality on transversally oriented Riemannian foliations. Also,
we show that there is an Atiyah-Singer type theorem for the equivariant
index of operators that are transversally elliptic with respect to a compact
Lie group action. This formula relies heavily on the stratification of the
manifold with group action and contains eta invariants and curvature forms.
These notes contain exercises at the end of each subsection and are meant to
be accessible to graduate students.
\end{abstract}

\maketitle

\section{Introduction to Ordinary Dirac Operators}

\subsection{The Laplacian}

The \textbf{Laplace operator} (or simply, \textbf{Laplacian}) is the famous
differential operator $\Delta $ on $\mathbb{R}^{n}$ defined by%
\begin{equation*}
\Delta h=-\sum_{j=1}^{n}\frac{\partial ^{2}h}{\partial x_{j}^{2}},h\in
C^{\infty }\left( \mathbb{R}^{n}\right)
\end{equation*}%
The solutions to the equation $\Delta h=0$ are the \textbf{harmonic functions%
}. This operator is present in both the heat equation and wave equations of
physics.%
\begin{eqnarray*}
\text{Heat equation}\text{: } &&\frac{\partial u\left( t,x\right) }{\partial
t}+\Delta _{x}u\left( t,x\right) =0 \\
\text{Wave equation} &\text{:}&\frac{\partial ^{2}u\left( t,x\right) }{%
\partial t^{2}}+\Delta _{x}u\left( t,x\right) =0
\end{eqnarray*}%
The sign of the Laplacian is chosen so that it is a nonnegative operator. If 
$\left\langle u,v\right\rangle $ denotes the $L^{2}$ inner product on
complex-valued functions on $\mathbb{R}^{n}$, by integrating by parts, we
see that 
\begin{eqnarray*}
\left\langle \Delta u,u\right\rangle &=&\int_{\mathbb{R}^{n}}\left( \Delta
u\right) \overline{u} \\
&=&\int_{\mathbb{R}^{n}}\left\vert \nabla u\right\vert ^{2}
\end{eqnarray*}%
if $u$ is compactly supported, where $\nabla u=\left( \frac{\partial u}{%
\partial x_{1}},...,\frac{\partial u}{\partial x_{n}}\right) $ is the
gradient vector. The calculation verifies the nonnegativity of $\Delta $.

The same result holds if instead the Laplace operator acts on the space of
smooth functions on a closed Riemannian manifold (compact, without
boundary); the differential operator is modified in a natural way to account
for the metric. That is, if the manifold is isometrically embedded in
Euclidean space, the Laplacian of a function on that manifold agrees with
the Euclidean Laplacian above if that function is extended to be constant in
the normal direction in a neighborhood of the embedded submanifold. One may
also define the Laplacian on differential forms in precisely the same way;
the Euclidean Laplacian on forms satisfies%
\begin{equation*}
\Delta \left[ u\left( x\right) dx_{i_{1}}\wedge dx_{i_{2}}\wedge ...\wedge
dx_{i_{p}}\right] =\left( \Delta u\right) \left( x\right) dx_{i_{1}}\wedge
dx_{i_{2}}\wedge ...\wedge dx_{i_{p}}.
\end{equation*}

These standard formulas for the Laplace operator suffice if the Riemannian
manifold is flat (for example, flat tori), but it is convenient to give a
coordinate-free description for this operator. If $\left( M,g\right) $ is a
smooth manifold with metric $g=\left( \bullet ,\bullet \right) $, the volume
form on $M$ satisfies $\mathrm{dvol}=\sqrt{\det g}dx$. The metric induces an
isomorphism $v_{p}\rightarrow v_{p}^{\flat }$ between vectors and one forms
at $p\in M$, given by%
\begin{equation*}
v_{p}^{\flat }\left( w_{p}\right) :=\left( v_{p},w_{p}\right) ,~w_{p}\in
T_{p}M.
\end{equation*}%
Thus, given an orthonormal basis $\left\{ e_{j}:1\leq j\leq n\right\} $of
the tangent space $T_{p}M$, we declare the corresponding dual basis $\left\{
e_{j}^{\flat }:1\leq j\leq n\right\} $ to be orthonormal, and in general we
declare $\left\{ e_{\alpha }^{\flat }=e_{\alpha _{1}}^{\flat }\wedge
...\wedge e_{\alpha _{r}}^{\flat }\right\} _{\left\vert \alpha \right\vert
=k}$ to be an orthonormal basis of $r$-forms at a point. Then the $L^{2}$
inner product of $r$-forms on $M$ is defined by%
\begin{equation*}
\left\langle \gamma ,\beta \right\rangle =\int_{M}\left( \gamma ,\beta
\right) ~\mathrm{dvol~}.
\end{equation*}

Next, if $d:\Omega ^{r}\left( M\right) \rightarrow \Omega ^{r+1}\left(
M\right) $ is the exterior derivative on smooth $r$-forms, we define $\delta
:\Omega ^{r+1}\left( M\right) \rightarrow \Omega ^{r}\left( M\right) $ to be
the formal adjoint of $d$ with respect to the $L^{2}$inner product. That is,
if $\omega \in \Omega ^{r+1}\left( M\right) $, we define $\delta \omega $ by
requiring 
\begin{equation*}
\left\langle \gamma ,\delta \omega \right\rangle =\left\langle d\gamma
,\omega \right\rangle ~
\end{equation*}%
for all $\gamma \in \Omega ^{r}\left( M\right) $. Then the \textbf{Laplacian
on differential }$r$\textbf{-forms on }$M$ is defined to be%
\begin{equation*}
\Delta =\delta d+d\delta :\Omega ^{r}\left( M\right) \rightarrow \Omega
^{r}\left( M\right) .
\end{equation*}%
It can be shown that $\Delta $ is an essentially self-adjoint operator. The
word \emph{essentially} means that the space of smooth forms needs to be
closed with respect to a certain Hilbert space norm, called a Sobolev norm.

We mention that in many applications, vector-valued Laplacians and
Laplacians on sections of vector bundles are used.

\begin{exercise}
Explicitly compute the formal adjoint $\delta $ for $d$ restricted to
compactly supported forms in Euclidean space, and verify that the $\delta
d+d\delta $ agrees with the Euclidean Laplacian on $r$-forms.
\end{exercise}

\begin{exercise}
Show that a smooth $r$-form $\alpha \in \Omega ^{r}\left( M\right) $ is
harmonic, meaning that $\Delta \alpha =0$, if and only if $d\alpha $ and $%
\delta \alpha $ are both zero.
\end{exercise}

\begin{exercise}
Explicitly compute the set of harmonic $r$-forms on the $2$-dimensional flat
torus $T^{2}=\mathbb{R}^{2}\diagup \mathbb{Z}^{2}$. Verify the \textbf{Hodge
Theorem} in this specific case; that is, show that the space of harmonic $r$%
-forms is isomorphic to the $r$ -dimensional de Rham cohomology group $%
H^{r}\left( M\right) $.
\end{exercise}

\begin{exercise}
Suppose that $\alpha $ is a representative of a cohomology class in $%
H^{r}\left( M\right) $. Show that $\alpha $ is a harmonic form if and only
if $\alpha $ is the element of the cohomology class with minimum $L^{2}$
norm.
\end{exercise}

\begin{exercise}
If $\left( g_{ij}\right) $ is the local matrix for the metric with $%
g_{ij}=\left( \frac{\partial }{\partial x_{i}},\frac{\partial }{\partial
x_{j}}\right) $, show that the matrix $\left( g^{ij}\right) $ defined by $%
g^{ij}=\left( dx_{i},dx_{j}\right) $ is the inverse of the matrix $\left(
g_{ij}\right) $.
\end{exercise}

\begin{exercise}
If $\alpha =\sum_{j=1}^{n}\alpha _{j}\left( x\right) dx_{j}$ is a one-form
on a Riemannian manifold of dimension $n$, where $g^{ij}=\left(
dx_{i},dx_{j}\right) $ is the local metric matrix for one-forms, verify that
the formal adjoint $\delta $ satisfies 
\begin{equation*}
\delta \left( \alpha \right) =-\frac{1}{\sqrt{g}}\sum_{i,j}\frac{\partial }{%
\partial x_{i}}\left( g^{ij}\sqrt{g}\alpha _{j}\right) .
\end{equation*}
\end{exercise}

\begin{exercise}
Show that if $f\in C^{\infty }\left( M\right) $, then 
\begin{equation*}
\int_{M}\Delta f~\mathrm{dvol}=0.
\end{equation*}
\end{exercise}

\subsection{The ordinary Dirac operator}

The original motivation for constructing a Dirac operator was the need of a
first-order differential operator whose square is the Laplacian. Dirac
needed such an operator in order to make some version of quantum mechanics
that is compatible with special relativity. Specifically, suppose that $%
D=\sum_{j=1}^{n}c_{j}\frac{\partial }{\partial x_{j}}$ is a first-order,
constant-coefficient differential operator on $\mathbb{R}^{n}$ such that $%
D^{2}$ is the ordinary Laplacian on $\mathbb{R}^{n}$. Then one is quickly
led to the equations%
\begin{eqnarray*}
c_{i}^{2} &=&-1,~ \\
c_{i}c_{j}+c_{j}c_{i} &=&0,~i\neq j
\end{eqnarray*}%
Clearly, this is impossible if we require each $c_{j}\in \mathbb{C}$.
However, if we allow matrix coefficients, we are able to find such matrices;
they are called \textbf{Clifford matrices}. In the particular case of $%
\mathbb{R}^{3}$, we may use the famous \textbf{Pauli spin matrices}%
\begin{equation*}
c_{1}=\left( 
\begin{array}{cc}
0 & i \\ 
i & 0%
\end{array}%
\right) ,~c_{2}=\left( 
\begin{array}{cc}
0 & 1 \\ 
-1 & 0%
\end{array}%
\right) ,~c_{3}=\left( 
\begin{array}{cc}
i & 0 \\ 
0 & -i%
\end{array}%
\right) .
\end{equation*}%
The vector space $\mathbb{C}^{k}$ on which the matrices and derivatives act
is called the vector space of spinors. It can be shown that the minimum
dimension $k$ satisfies $k=2^{\left\lfloor n/2\right\rfloor }$. The matrices
can be used to form an associated \textbf{Clifford multiplication of vectors}%
, written $c\left( v\right) $, defined by 
\begin{equation*}
c\left( v\right) =\sum_{j=1}^{n}v_{j}c_{j}~,
\end{equation*}%
where $v=\left( v_{1},...,v_{n}\right) $. Note that $c:\mathbb{R}%
^{n}\rightarrow \mathrm{End}\left( \mathbb{C}^{k}\right) $ such that 
\begin{equation*}
c\left( v\right) c\left( w\right) +c\left( w\right) c\left( v\right)
=-2\left( v,w\right) ,~v,w\in \mathbb{R}^{n}.
\end{equation*}

If $M$ is a closed Riemannian manifold, we desire to find a Hermitian vector
bundle $E\rightarrow M$ and a first-order differential operator $D:\Gamma
\left( E\right) \rightarrow \Gamma \left( E\right) $ on sections of $E$ such
that its square is a Laplacian plus a lower-order differential operator.
This implies that each $E_{x}$ is a $\mathbb{C}\mathrm{l}\left(
T_{x}M\right) $-module, where $\mathbb{C}\mathrm{l}\left( T_{x}M\right) $ is
the subalgebra of $\mathrm{End}_{\mathbb{C}}\left( E_{x}\right) $ generated
by a Clifford multiplication of tangent vectors. Then the Dirac operator
associated to the Clifford module $E$ is defined for a local orthonormal
frame $\left( e_{j}\right) _{j=1}^{n}$ of $TM$ to be%
\begin{equation*}
D=\sum_{j=1}^{n}c\left( e_{j}\right) \nabla _{e_{j}},
\end{equation*}%
where $c$ denotes Clifford multiplication and where $\nabla $ is a metric
connection on $E$ satisfying the compatibility condition%
\begin{equation*}
\nabla _{V}\left( c\left( W\right) s\right) =c\left( \nabla _{V}W\right)
s+c\left( W\right) \nabla _{V}s
\end{equation*}%
for all sections $s\in \Gamma \left( E\right) $ and vector fields $V,W\in
\Gamma \left( TM\right) $. We also require that Clifford multiplication of
vectors is skew-adjoint with respect to the $L^{2}$ inner product, meaning
that 
\begin{equation*}
\left\langle c\left( v\right) s_{1},s_{2}\right\rangle =-\left\langle
s_{1},c\left( v\right) s_{2}\right\rangle
\end{equation*}%
for all $v\in \Gamma \left( TM\right) $, $s_{1},s_{2}\in \Gamma \left(
E\right) $. It can be shown that the expression for $D$ above is independent
of the choice of orthonormal frame of $TM$. In the case where $E$ has the
minimum possible rank $k=2^{\left\lfloor n/2\right\rfloor }$, we call $E$ a 
\textbf{complex spinor bundle} and $D$ a \textbf{spin}$^{c}$\textbf{\ Dirac
operator}. If such a bundle exists over a smooth manifold $M$, we say that $%
M $ is spin$^{c}$. There is a mild topological obstruction to the existence
of such a structure; the third integral Stiefel-Whitney class of $TM$ must
vanish.

Often the bundle $E$ comes equipped with a grading $E=E^{+}\oplus E^{-}$
such that $D$ maps $\Gamma \left( E^{+}\right) $ to $\Gamma \left(
E^{-}\right) $ and vice-versa. In these cases, we often restrict our
attention to $D:\Gamma \left( E^{+}\right) \rightarrow \Gamma \left(
E^{-}\right) $

Examples of ordinary Dirac operators are as follows:

\begin{itemize}
\item The \textbf{de Rham operator} is defined to be%
\begin{equation*}
d+\delta :\Omega ^{\text{\textrm{even}}}\left( M\right) \rightarrow \Omega ^{%
\mathrm{odd}}\left( M\right)
\end{equation*}%
from even forms to odd forms. In this case, the Clifford multiplication is
given by $c\left( v\right) =v^{\flat }\wedge -i\left( v\right) $, where $%
v\in T_{x}M$ and $i\left( v\right) $ denotes interior product, and $\nabla $
is the ordinary Levi-Civita connection extended to forms.

\item If $M$ is even-dimensional, the \textbf{signature operator} is defined
to be%
\begin{equation*}
d+\delta :\Omega ^{+}\left( M\right) \rightarrow \Omega ^{-}\left( M\right)
\end{equation*}%
from \textbf{self-dual} to \textbf{anti-self-dual} forms. This grading is
defined as follows. Let $\ast $ denote the \textbf{Hodge star operator} on
forms, defined as the unique endomorphism of the bundle of forms such that $%
\ast :\Omega ^{r}\left( M\right) \rightarrow \Omega ^{n-r}\left( M\right) $
and 
\begin{equation*}
\alpha \wedge \ast \beta =\left( \alpha ,\beta \right) \mathrm{dvol},~\alpha
,\beta \in \Omega ^{r}\left( M\right) .
\end{equation*}%
Then observe that the operator%
\begin{equation*}
\bigstar =i^{r\left( r-1\right) +\frac{n}{2}}\ast :\Omega ^{r}\left(
M\right) \rightarrow \Omega ^{n-r}\left( M\right)
\end{equation*}%
satisfies $\bigstar ^{2}=1$. Then it can be shown that $d+\delta $
anticommutes with $\bigstar $ and thus maps the $+1$ eigenspace of $\bigstar 
$, denoted $\Omega ^{+}\left( M\right) $, to the $-1$ eigenspace of $%
\bigstar $, denoted $\Omega ^{-}\left( M\right) $. Even though the bundles
have changed from the previous example, the expression for Clifford
multiplication is the same.

\item If $M$ is complex, then the \textbf{Dolbeault operator} is defined to
be%
\begin{equation*}
\overline{\partial }+\overline{\partial }^{\ast }:\Omega ^{0,\text{\textrm{%
even}}}\left( M\right) \rightarrow \Omega ^{0,\mathrm{odd}}\left( M\right) ,
\end{equation*}%
where the differential forms involve wedge products of $d\overline{z_{j}}$
and the differential $\overline{\partial }$ differentiates only with respect
to the $\overline{z_{j}}$ variables.

\item The \textbf{spin}$^{c}$\textbf{\ Dirac operator} has already been
mentioned above. The key point is that the vector bundle $S\rightarrow M$ in
this case has the minimum possible dimension. When $M$ is even-dimensional,
then the spinor bundle decomposes as $S^{+}\oplus S^{-}$, and $D:\Gamma
\left( S^{+}\right) \rightarrow \Gamma \left( S^{-}\right) $. The spinor
bundle $S$ is unique up to tensoring with a complex line bundle.
\end{itemize}

For more information on Dirac operators, spin manifolds, and Clifford
algebras, we refer the reader to \cite{LM} and \cite{Roe}. Often the
operators described above are called \emph{Dirac-type operators}, with the
word \textquotedblleft Dirac operator\textquotedblright\ reserved for the
special examples of the spin or spin$^{c}$ Dirac operator. Elements of the
kernel of a spin or spin$^{c}$ Dirac operator are called \textbf{harmonic
spinors}.

\begin{exercise}
Let the Dirac operator $D$ on the two-dimensional torus $T^{2}=\mathbb{R}%
^{2}\diagup \mathbb{Z}^{2}$ be defined using $c_{1}$ and $c_{2}$ of the
Pauli spin matrices. Find a decomposition of the bundle as $S^{+}\oplus
S^{-} $, and calculate $\ker \left( \left. D\right\vert _{S^{+}}\right) $
and $\ker \left( \left. D\right\vert _{S^{-}}\right) $. Find all the
eigenvalues and corresponding eigensections of $D^{+}=\left. D\right\vert
_{S^{+}}$.
\end{exercise}

\begin{exercise}
On an $n$-dimensional manifold $M$, show that $\ast ^{2}=\left( -1\right)
^{r\left( n-r\right) }$ and $\bigstar ^{2}=1$ when restricted to $r$-forms.
\end{exercise}

\begin{exercise}
On $\mathbb{R}^{4}$ with metric $ds^{2}=dx_{1}^{2}+4dx_{2}^{2}+dx_{3}^{2}+%
\left( 1+\exp \left( x_{1}\right) \right) ^{2}dx_{4}^{2}$, let $\omega
=x_{1}^{2}x_{2}dx_{2}\wedge dx_{4}$. Find $\ast \omega $ and $\bigstar
\omega $.
\end{exercise}

\begin{exercise}
Calculate the signature operator on $T^{2}$, and identify the subspaces $%
\Omega ^{+}\left( T^{2}\right) $ and $\Omega ^{-}\left( T^{2}\right) $.
\end{exercise}

\begin{exercise}
Show that $-i\frac{\partial }{\partial \theta }$ is a Dirac operator on $%
S^{1}=\left\{ e^{i\theta }:\theta \in \mathbb{R}\right\} $ . Find all the
eigenvalues and eigenfunctions of this operator.
\end{exercise}

\begin{exercise}
Show that if $S$ and $T$ are two anticommuting linear transformations from a
vector space $V$ to itself, then if $E_{\lambda }$ is the eigenspace of $S$
corresponding to an eigenvalue $\lambda $, then $TE_{\lambda }$ is the
eigenspace of $S$ corresponding to the eigenvalue $-\lambda $.
\end{exercise}

\begin{exercise}
Show that if $\delta ^{r}$ is the adjoint of $d:\Omega ^{r-1}\left( M\right)
\rightarrow \Omega ^{r}\left( M\right) $, then%
\begin{equation*}
\delta ^{r}=\left( -1\right) ^{nr+n+1}\ast d\ast
\end{equation*}%
on $\Omega ^{r}\left( M\right) $.
\end{exercise}

\begin{exercise}
Show that if the dimension of $M$ is even and $\delta ^{r}$ is the adjoint
of $d:\Omega ^{r-1}\left( M\right) \rightarrow \Omega ^{r}\left( M\right) $,
then%
\begin{equation*}
\delta ^{r}=-\bigstar d\bigstar
\end{equation*}%
on $\Omega ^{r}\left( M\right) $. Is this true if the dimension is odd?
\end{exercise}

\begin{exercise}
Show that $d+\delta $ maps $\Omega ^{+}\left( M\right) $ to $\Omega
^{-}\left( M\right) $.
\end{exercise}

\begin{exercise}
Show that if we write the Dirac operator for $\mathbb{R}^{3}$ 
\begin{equation*}
D=c_{1}\frac{\partial }{\partial x_{1}}+c_{2}\frac{\partial }{\partial x_{2}}%
+c_{3}\frac{\partial }{\partial x_{3}}
\end{equation*}%
using the Pauli spin matrices in geodesic polar coordinates%
\begin{equation*}
D=Z\left( \frac{\partial }{\partial r}+D^{S}\right) ,
\end{equation*}%
then $ZD^{S}$ restricts to a spin$^{c}$ Dirac operator on the unit sphere $%
S^{2}$, and $Z$ is Clifford multiplication by the vector $\frac{\partial }{%
\partial r}$.
\end{exercise}

\begin{exercise}
Show that $d+\delta =\sum_{j=1}^{n}c\left( e_{j}\right) \nabla _{e_{j}}$,
with the definition of Clifford multiplication given in the notes.
\end{exercise}

\begin{exercise}
Show that the expression $\sum_{j=1}^{n}c\left( e_{j}\right) \nabla _{e_{j}}$
for the Dirac operator is independent of the choice of orthonormal frame.
\end{exercise}

\subsection{\protect\vspace{1pt}Properties of Dirac operators}

Here we describe some very important properties of Dirac operators.

First, Dirac operators are \emph{elliptic}. Both the Laplacian and Dirac
operators are examples of such operators. Very roughly, the word \emph{%
elliptic} means that the operators differentiate in all possible directions.
To state more precisely what this means, we need to discuss what is called
the \emph{principal symbol} of a differential (or pseudodifferential)
operator.

Very roughly, the principal symbol is the set of matrix-valued leading order
coefficients of the operator. If $E\rightarrow M$, $F\rightarrow M$ are two
vector bundles and $P:\Gamma \left( E\right) \rightarrow \Gamma \left(
F\right) $ is a differential operator of order $k$ acting on sections, then
in local coordinates of a local trivialization of the vector bundles, we may
write 
\begin{equation*}
P=\sum_{\left\vert \alpha \right\vert =k}s_{\alpha }\left( x\right) \frac{%
\partial ^{k}}{\partial x^{\alpha }}+\text{lower order terms},
\end{equation*}%
where the sum is over all possible multi-indices $\alpha =\left( \alpha
_{1},...,\alpha _{k}\right) $ of length $\left\vert \alpha \right\vert =k$,
and each $s_{\alpha }\left( x\right) \in \mathrm{\mathrm{Hom}}\left(
E_{x},F_{x}\right) $ is a linear transformation. If $\xi =\sum \xi
_{j}dx_{j}\in T_{x}^{\ast }M$ is a nonzero covector at $x$, we define the 
\textbf{principal symbol} $\sigma \left( P\right) \left( \xi \right) $ of $P$
at $\xi $ to be%
\begin{equation*}
\sigma \left( P\right) \left( \xi \right) =i^{k}\sum_{\left\vert \alpha
\right\vert =k}s_{\alpha }\left( x\right) \xi ^{\alpha }~\in \mathrm{\mathrm{%
Hom}}\left( E_{x},F_{x}\right) ,
\end{equation*}%
with $\xi ^{\alpha }=\xi _{\alpha _{1}}\xi _{\alpha _{2}}...\xi _{\alpha
_{k}}$ (some people leave out the $i^{k}$). It turns out that by defining it
this way, it is invariant under coordinate transformations. One
coordinate-free definition of $\sigma \left( P\right) _{x}:T_{x}^{\ast
}\left( M\right) \rightarrow \mathrm{\mathrm{Hom}}\left( E_{x},F_{x}\right) $
is as follows. For any $\xi \in T_{x}^{\ast }\left( M\right) $, choose a
locally-defined function $f$ such that $df_{x}=\xi $. Then we define the
operator%
\begin{equation*}
\sigma _{m}\left( P\right) \left( \xi \right) =\lim_{t\rightarrow \infty }%
\frac{1}{t^{m}}\left( e^{-itf}Pe^{itf}\right) ,
\end{equation*}%
where $\left( e^{-itf}Pe^{itf}\right) \left( u\right) =e^{-itf}\left(
P\left( e^{itf}u\right) \right) $. Then the order $k$ of the operator and
principal symbol are defined to be%
\begin{eqnarray*}
k &=&\sup \left\{ m:\sigma _{m}\left( P\right) \left( \xi \right) <\infty
\right\} \\
\sigma \left( P\right) \left( \xi \right) &=&\sigma _{k}\left( P\right)
\left( \xi \right) ~.
\end{eqnarray*}%
With this definition, the principal symbol of any differential (or even
pseudodifferential) operator can be found. Pseudodifferential operators are
more general operators that can be defined locally using the Fourier
transform and include such operators as the square root of the Laplacian.

An \textbf{elliptic differential (or pseudodifferential) operator} $P$ on $M$
is defined to be an operator such that its principal symbol $\sigma \left(
P\right) \left( \xi \right) $ is invertible for all nonzero $\xi \in T^{\ast
}M$.

From the exercises at the end of this section, we see that the symbol of any
Dirac operator $D=\sum c\left( e_{j}\right) \nabla _{e_{j}}$ is%
\begin{equation*}
\sigma \left( D\right) \left( \xi \right) =ic\left( \xi ^{\#}\right) \text{,}
\end{equation*}%
and the symbol of the associated Laplacian $D^{2}$ is%
\begin{equation*}
\sigma \left( D^{2}\right) =\left( ic\left( \xi ^{\#}\right) \right)
^{2}=\left\Vert \xi \right\Vert ^{2},
\end{equation*}%
which is clearly invertible for $\xi \neq 0$. Therefore both $D$ and $D^{2}$
are elliptic.

We say that an operator $P$ is \textbf{strongly elliptic} if there exists $%
c>0$ such that 
\begin{equation*}
\sigma \left( D\right) \left( \xi \right) \geq c\left\vert \xi \right\vert
^{2}
\end{equation*}%
for all nonzero $\xi \in T^{\ast }M$. The Laplacian and $D^{2}$ are strongly
elliptic.

Following are important properties of elliptic operators $P$, which now
apply to Dirac operators and their associated Laplacians:

\begin{itemize}
\item \textbf{Elliptic regularity}: if the coefficients of $P$ are smooth,
then if $Pu$ is smooth, then $u$ is smooth. As a consequence, if the order
of $P$ is greater than zero, then the kernel and all other eigenspaces of $P$
consist of smooth sections.

\item Elliptic operators are Fredholm when the correct Sobolev spaces of
sections are used.

\item Ellipticity implies that the spectrum of $P$ consists of eigenvalues.
Strong ellipticity implies that the spectrum is discrete and has the only
limit point at infinity. In particular, the eigenspaces are
finite-dimensional and consist of smooth sections. This now applies to any
Dirac operator, because its square is strongly elliptic.

\item If $P$ is a second order elliptic differential operator with no zero$^{%
\mathrm{th}}$ order terms, strong ellipticity implies the maximum principle
for the operator $P$.

\item Many inequalities for elliptic operators follow, like G\.{a}rding's
inequality, the elliptic estimates, etc.
\end{itemize}

See \cite{Wells}, \cite{Roe}, and \cite{Shu} for more information on
elliptic differential and pseudodifferential operators on manifolds.

Next, any Dirac operator $D:\Gamma \left( E\right) \rightarrow \Gamma \left(
E\right) $ is \textbf{formally self-adjoint}, meaning that when restricted
to smooth compactly-supported sections $u,v\in \Gamma \left( E\right) $ it
satisfies%
\begin{equation*}
\left\langle Du,v\right\rangle =\left\langle u,Dv\right\rangle .
\end{equation*}%
Since $D$ is elliptic and if $M$ is closed, it then follows that $D$ is
essentially self-adjoint, meaning that there is a Hilbert space $H^{1}\left(
E\right) $ such that $\Gamma \left( E\right) \subset H^{1}\left( E\right)
\subset L^{2}\left( E\right) $ such that the closure of $D$ in $H^{1}\left(
E\right) $ is a truly self-adjoint operator defined on the whole space. In
this particular case, $H^{1}\left( E\right) $ is an example of a Sobolev
space, which is the closure of $\Gamma \left( E\right) $ with respect to the
norm $\left\Vert u\right\Vert _{1}=\left\Vert u\right\Vert +\left\Vert
Du\right\Vert $, where $\left\Vert \bullet \right\Vert $ denotes the
ordinary $L^{2}$-norm.

We now show the proof that $D$ is formally self-adjoint. If the local bundle
inner product on $E$ is $\left( \bullet ,\bullet \right) $, we have%
\begin{eqnarray*}
\left( Du,v\right) &=&\sum \left( c\left( e_{j}\right) \nabla
_{e_{j}}u,v\right) =\sum -\left( \nabla _{e_{j}}u,c\left( e_{j}\right)
v\right) \\
&=&\sum -e_{j}\left( u,c\left( e_{j}\right) v\right) +\left( u,\nabla
_{e_{j}}\left( c\left( e_{j}\right) v\right) \right) ,
\end{eqnarray*}%
since $c\left( e_{j}\right) $ is skew-adjoint and $\nabla $ is a metric
connection. Using the compatibility of the connection, we have%
\begin{equation*}
\left( Du,v\right) =\sum -e_{j}\left( u,c\left( e_{j}\right) v\right)
+\left( u,c\left( \nabla _{e_{j}}e_{j}\right) v\right) +\left( u,c\left(
e_{j}\right) \nabla _{e_{j}}v\right) .
\end{equation*}%
Next, we use the fact that we are allowed to choose the local orthonormal
frame in any way we wish. If we are evaluating this local inner product at a
point $x\in M$, we choose the orthonormal frame $\left( e_{i}\right) $ so
that all covariant derivatives of $e_{i}$ vanish at $x$. Now, the middle
term above vanishes, and%
\begin{equation*}
\left( Du,v\right) =\left( u,Dv\right) +\sum -e_{j}\left( u,c\left(
e_{j}\right) v\right) .
\end{equation*}%
Next, if $\omega $ denotes the one-form defined by $\omega \left( X\right)
=\left( u,c\left( X\right) v\right) $ for $X\in \Gamma \left( TM\right) $,
then an exercise at the end of this section implies that 
\begin{equation*}
\left( \delta \omega \right) \left( x\right) =\left( \sum -e_{j}\left(
u,c\left( e_{j}\right) v\right) \right) \left( x\right) ,
\end{equation*}%
with our choice of orthonormal frame. Hence, 
\begin{equation*}
\left( Du,v\right) =\left( u,Dv\right) +\delta \omega ,
\end{equation*}%
which is a general formula now valid at all points of $M$. After integrating
over $M$ we have%
\begin{eqnarray*}
\left\langle Du,v\right\rangle &=&\left\langle u,Dv\right\rangle
+\int_{M}\delta \omega ~\mathrm{dvol} \\
&=&\left\langle u,Dv\right\rangle +\int_{M}\left( d\left( 1\right) ,\omega
\right) ~\mathrm{dvol} \\
&=&\left\langle u,Dv\right\rangle .
\end{eqnarray*}%
Thus, $D$ is formally self-adjoint.

\begin{exercise}
Find the principal symbol of the wave operator $\frac{\partial ^{2}}{%
\partial t^{2}}-\frac{\partial ^{2}}{\partial x^{2}}$ on $\mathbb{R}^{2}$,
and determine if it is elliptic.
\end{exercise}

\begin{exercise}
If $P_{1}$ and $P_{2}$ are two differential operators such that the
composition $P_{1}P_{2}$ is defined, show that 
\begin{equation*}
\sigma \left( P_{1}P_{2}\right) \left( \xi \right) =\sigma \left(
P_{1}\right) \left( \xi \right) \sigma \left( P_{2}\right) \left( \xi
\right) .
\end{equation*}
\end{exercise}

\begin{exercise}
Prove that if $D=\sum c\left( e_{j}\right) \nabla _{e_{j}}$ is a Dirac
operator, then%
\begin{equation*}
\sigma \left( D\right) \left( \xi \right) =ic\left( \xi ^{\#}\right) \text{
and }\sigma \left( D^{2}\right) \left( \xi \right) =\left\Vert \xi
\right\Vert ^{2}
\end{equation*}%
for all $\xi \in T^{\ast }M$. Use the coordinate-free definition.
\end{exercise}

\begin{exercise}
Show that if $\omega $ is a one-form on $M$, then 
\begin{equation*}
\left( \delta \omega \right) \left( x\right) =-\left(
\sum_{j=1}^{n}e_{j}\left( \omega \left( e_{j}\right) \right) \right) \left(
x\right) ,
\end{equation*}%
if $\left( e_{1},...,e_{n}\right) $ is a local orthonormal frame of $TM$
chosen so that 
\begin{equation*}
\left( \nabla _{e_{j}}e_{k}\right) \left( x\right) =0
\end{equation*}%
at $x\in M$, for every $j,k\in \left\{ 1,...,n\right\} $.
\end{exercise}

\subsection{The Atiyah-Singer Index Theorem}

Given Banach spaces $S$ and $T$, a bounded linear operator $L:S\rightarrow T$
is called \textbf{Fredholm} if its range is closed and its kernel and
cokernel $T\diagup L\left( S\right) $ are finite dimensional. The \textbf{%
index} of such an operator is defined to be 
\begin{equation*}
\mathrm{ind}\left( L\right) =\dim \ker \left( L\right) -\dim \text{\textrm{%
coker}}\left( L\right) ,
\end{equation*}%
and this index is constant on continuous families of such $L$. In the case
where $S$ and $T$ are Hilbert spaces, this is the same as 
\begin{equation*}
\mathrm{ind}\left( L\right) =\dim \ker \left( L\right) -\dim \mathrm{\ker }%
\left( L^{\ast }\right) .
\end{equation*}%
The index determines the connected component of $L$ in the space of Fredholm
operators. We will be specifically interested in this integer in the case
where $L$ is a Dirac operator.

For the case of the de Rham operator, we have 
\begin{equation*}
\ker \left( d+\delta \right) =\ker \left( d+\delta \right) ^{2}=\ker \Delta ,
\end{equation*}%
so that 
\begin{equation*}
\mathcal{H}^{r}\left( M\right) =\ker \left( \left. \left( d+\delta \right)
\right\vert _{\Omega ^{r}}\right)
\end{equation*}%
is the space of harmonic forms of degree $r$, which by the Hodge theorem is
isomorphic to $H^{r}\left( M\right) $, the $r^{\mathrm{th}}$ de Rham
cohomology group. Therefore, index 
\begin{eqnarray*}
\mathrm{ind}\left( \left. \left( d+\delta \right) \right\vert _{\Omega ^{%
\mathrm{even}}}\right) &:&=\dim \ker \left( \left. \left( d+\delta \right)
\right\vert _{\Omega ^{\mathrm{even}}}\right) -\dim \ker \left( \left.
\left( d+\delta \right) \right\vert _{\Omega ^{\mathrm{even}}}\right) ^{\ast
} \\
&=&\dim \ker \left( \left. \left( d+\delta \right) \right\vert _{\Omega ^{%
\mathrm{even}}}\right) -\dim \ker \left( \left. \left( d+\delta \right)
\right\vert _{\Omega ^{\mathrm{odd}}}\right) \\
&=&\chi \left( M\right) ,
\end{eqnarray*}%
the Euler characteristic of $M$.

In general, suppose that $D$ is an elliptic operator of order $m$ on
sections of a vector bundle $E^{\pm }$ over a smooth, compact manifold $M$.
Let $H^{s}\left( \Gamma \left( M,E^{\pm }\right) \right) $ denote the
Sobolev $s$-norm completion of the space of sections $\Gamma \left(
M,E\right) $, with respect to a chosen metric. Then $D$ can be extended to
be a bounded linear operator $\overline{D_{s}}:H^{s}\left( \Gamma \left(
M,E^{+}\right) \right) \rightarrow H^{s-m}\left( \Gamma \left(
M,E^{-}\right) \right) $ that is Fredholm, and $\mathrm{ind}\left( D\right)
:=\mathrm{ind}\left( \overline{D_{s}}\right) $ is well-defined and
independent of $s$. In the 1960s, the researchers M. F. Atiyah and I. Singer
proved that the index of an elliptic operator on sections of a vector bundle
over a smooth manifold satisfies the following formula (\cite{AS0}, \cite%
{AS1}):%
\begin{eqnarray*}
\mathrm{ind}\left( D\right) &=&\int_{M}\mathrm{ch}\left( \sigma \left(
D\right) \right) \wedge \mathrm{Todd}\left( T_{\mathbb{C}}M\right) \\
&=&\int_{M}\alpha \left( x\right) ~\mathrm{dvol}\left( x\right) ,
\end{eqnarray*}%
where $\mathrm{ch}\left( \sigma \left( D\right) \right) $ is a form
representing the Chern character of the principal symbol $\sigma \left(
D\right) $, and $\mathrm{Todd}\left( T_{\mathbb{C}}M\right) $ is a form
representing the Todd class of the complexified tangent bundle $T_{\mathbb{C}%
}M$; these forms are characteristic forms derived from the theory of
characteristic classes and depend on geometric and topological data. The
local expression for the relevant term of the integrand, which is a multiple
of the volume form $\mathrm{dvol}\left( x\right) $, can be written in terms
of curvature and the principal symbol and is denoted $\alpha \left( x\right)
~\mathrm{dvol}\left( x\right) $.

Typical examples of this theorem are some classic theorems in global
analysis. As in the earlier example, let $D=d+\delta $ from the space of
even forms to the space of odd forms on the manifold $M$ of dimension $n$,
where as before $\delta $ denotes the $L^{2}$-adjoint of the exterior
derivative $d$. Then the elements of $\ker \left( d+\delta \right) $ are the
even harmonic forms, and the elements of the cokernel can be identified with
odd harmonic forms. Moreover,%
\begin{eqnarray*}
\mathrm{ind}\left( d+\delta \right) &=&\dim H^{\mathrm{even}}\left( M\right)
-\dim H^{\mathrm{odd}}\left( M\right) \\
&=&\chi \left( M\right) ,\text{ and} \\
\int_{M}\mathrm{ch}\left( \sigma \left( d+\delta \right) \right) \wedge 
\mathrm{Todd}\left( T_{\mathbb{C}}M\right) &=&\frac{1}{\left( 2\pi \right)
^{n}}\int_{M}\mathrm{Pf~},
\end{eqnarray*}%
where $\mathrm{Pf}$ is the Pfaffian, which is, suitably interpreted, a
characteristic form obtained using the square root of the determinant of the
curvature matrix. In the case of $2$-manifolds ($n=2$), $\mathrm{Pf}$ is the
Gauss curvature times the area form. Thus, in this case the Atiyah-Singer
Index Theorem yields the generalized Gauss-Bonnet Theorem.

Another example is the operator $D=d+d^{\ast }$ on forms on an
even-dimensional manifold, this time mapping the self-dual to the
anti-self-dual forms. This time the Atiyah-Singer Index Theorem yields the
equation (called the Hirzebruch Signature Theorem)%
\begin{equation*}
\mathrm{Sign}\left( M\right) =\int_{M}L,
\end{equation*}%
where Sign$\left( M\right) $ is signature of the manifold, and $L$ is the
Hirzebruch $L$-polynomial applied to the Pontryagin forms.

If a manifold is spin, then the index of the spin Dirac operator is the $%
\widehat{A}$ genus (\textquotedblleft $A$-roof\textquotedblright\ genus) of
the manifold. Note that the spin Dirac operator is an example of a spin$^{c}$
Dirac operator where the spinor bundle is associated to a principal $\mathrm{%
Spin}\left( n\right) $ bundle. Such a structure exists when the second
Stiefel-Whitney class is zero, a stronger condition than the spin$^{c}$
condition. The $\widehat{A}$ genus is normally a rational number but must
agree with the index when the manifold is spin.

Different examples of operators yield other classical theorems, such as the
Hirzebruch-Riemann-Roch Theorem, which uses the Dolbeault operator.

All of the first order differential operators mentioned above are examples
of Dirac operators. If $M$ is spin$^{c}$, then the Atiyah-Singer Index
Theorem reduces to a calculation of the index of Dirac operators (twisted by
a bundle). Because of this and the Thom isomorphism in $K$-theory, the Dirac
operators and their symbols play a very important role in proofs of the
Atiyah-Singer Index Theorem. For more information, see \cite{AS1}, \cite{LM}.

\begin{exercise}
Prove that if $L:\mathcal{H}_{1}\rightarrow \mathcal{H}_{2}$ is a Fredholm
operator between Hilbert spaces, then \textrm{coker}$\left( L\right) \cong
\ker \left( L^{\ast }\right) $.
\end{exercise}

\begin{exercise}
Suppose that $P:\mathcal{H}\rightarrow \mathcal{H}$ is a self-adjoint linear
operator, and $\mathcal{H}=\mathcal{H}^{+}\oplus \mathcal{H}^{-}$ is an
orthogonal decomposition. If $P$ maps $\mathcal{H}^{+}$ into $\mathcal{H}%
^{-} $ and vice versa, prove that the adjoint of the restriction $P:\mathcal{%
H}^{+}\rightarrow \mathcal{H}^{-}$ is the restriction of $P$ to $\mathcal{H}%
^{-}$. Also, find the adjoint of the operator $P^{\prime }:\mathcal{H}%
^{+}\rightarrow \mathcal{H}$ defined by $P^{\prime }\left( h\right) =P\left(
h\right) $.
\end{exercise}

\begin{exercise}
If $D$ is an elliptic operator and $E_{\lambda }$ is an eigenspace of $%
D^{\ast }D$ corresponding to the eigenvalue $\lambda \neq 0$, then $D\left(
E_{\lambda }\right) $ is the eigenspace of $DD^{\ast }$ corresponding to the
eigenvalue $\lambda $. Conclude that the eigenspaces of $D^{\ast }D$ and $%
DD^{\ast }$ corresponding to nonzero eigenvalues have the same (finite)
dimension.
\end{exercise}

\begin{exercise}
Let $f:\mathbb{C}\rightarrow \mathbb{C}$ be a smooth function, and let $L:%
\mathcal{H}\rightarrow \mathcal{H}$ be a self-adjoint operator with discrete
spectrum. Let $P_{\lambda }:\mathcal{H}\rightarrow E_{\lambda }$ be the
orthogonal projection to the eigenspace corresponding to the eigenvalue $%
\lambda $. We define the operator $f\left( L\right) $ to be%
\begin{equation*}
f\left( L\right) =\sum_{\lambda }f\left( \lambda \right) P_{\lambda },
\end{equation*}%
assuming the right hand side converges. Assuming $f\left( L\right) $, $%
g\left( L\right) $, and $f\left( L\right) g\left( L\right) $ converge, prove
that $f\left( L\right) g\left( L\right) =g\left( L\right) f\left( L\right) $%
. Also, find the conditions on a function $f$ such that $f\left( L\right) =L$%
.
\end{exercise}

\begin{exercise}
Show that if $P$ is a self-adjoint Fredholm operator, then%
\begin{equation*}
\mathrm{ind}\left( D\right) =\mathrm{tr}\left( \exp \left( -tD^{\ast
}D\right) \right) -\mathrm{tr}\left( \exp \left( -tDD^{\ast }\right) \right)
\end{equation*}%
for all $t>0$, assuming that $\exp \left( -tD^{\ast }D\right) $, $\exp
\left( -tDD^{\ast }\right) $, and their traces converge.
\end{exercise}

\begin{exercise}
Find all homeomorphism types of surfaces $S$ such that a metric $g$ on $S$
has Gauss curvature $K_{g}$ that satisfies $-5\leq K_{g}\leq 0$ and volume
that satisfies $1\leq $\textrm{Vol}$_{g}\left( S\right) \leq 4$.
\end{exercise}

\begin{exercise}
Find an example of a smooth closed manifold $M$ such that every possible
metric on $M$ must have nonzero $L$ (the Hirzebruch $L$-polynomial applied
to the Pontryagin forms).
\end{exercise}

\begin{exercise}
Suppose that on a certain manifold the $\widehat{A}$-genus is $\frac{3}{4}$.
What does this imply about Stiefel-Whitney classes?
\end{exercise}

\section{Transversal Dirac operators on distributions\label%
{transvDiracDistrSection}}

This section contains some of the results in \cite{PrRi}, joint work with I.
Prokhorenkov.

The main point of this section is to provide some ways to analyze operators
that are not elliptic but behave in some ways like elliptic operators on
sections that behave nicely with respect to a designated \emph{transverse
subbundle} $Q\subseteq TM$. A \textbf{transversally elliptic differential
(or pseudodifferential) operator} $P$ on $M$ with respect to the transverse
distribution $Q\subseteq TM$ is defined to be an operator such that its
principal symbol $\sigma \left( P\right) \left( \xi \right) $ is required to
be invertible only for all nonzero $\xi \in Q^{\ast }\subseteq T^{\ast }M$.
In later sections, we will be looking at operators that are transversally
elliptic with respect to the orbits of a group action, and in this case $Q$
is the normal bundle to the orbits, which may have different dimensions at
different points of the manifold. In this section, we will restrict to the
case where $Q$ has constant rank.

Now, let $Q\subset TM$ be a smooth distribution, meaning that $Q\rightarrow
M $ is a smooth subbundle of the tangent bundle. Assume that a $\mathbb{C}%
\mathrm{l}\left( Q\right) $-module structure on a complex Hermitian vector
bundle $E$ is given, and we will now define transverse Dirac operators on
sections of $E$. Similar to the above, $M$ is a closed Riemannian manifold, $%
c:Q\rightarrow \mathrm{End}\left( E\right) $ is the Clifford multiplication
on $E$, and $\nabla ^{E}$ is a $\mathbb{C}\mathrm{l}\left( Q\right) $
connection that is compatible with the metric on $M$; that is, Clifford
multiplication by each vector is skew-Hermitian, and we require%
\begin{equation*}
\nabla _{X}^{E}\left( c\left( V\right) s\right) =c\left( \nabla
_{X}^{Q}V\right) s+c\left( V\right) \nabla _{X}^{E}s
\end{equation*}%
for all $X\in \Gamma \left( TM\right) $, $V\in \Gamma Q$, and $s\in \Gamma E$%
.

\begin{remark}
For a given distribution $Q\subset TM$, it is always possible to obtain a
bundle of $\mathbb{C}\mathrm{l}\left( Q\right) $-modules with Clifford
connection from a bundle of $\mathbb{C}\mathrm{l}\left( TM\right) $-Clifford
modules, but not all such $\mathbb{C}\mathrm{l}\left( Q\right) $ connections
are of that type.
\end{remark}

Let $L=Q^{\bot }$, let $\left( f_{1},...,f_{q}\right) $ be a local
orthonormal frame for $Q$, and let $\pi :TM\rightarrow Q$ be the orthogonal
projection. We define the Dirac operator $A_{Q}$ corresponding to the
distribution $Q$ as 
\begin{equation}
A_{Q}=\sum_{j=1}^{q}c\left( f_{j}\right) \nabla _{f_{j}}^{E}.  \label{AQdef}
\end{equation}%
This definition is again independent of the choice of orthonormal frame; in
fact it is the composition of the maps 
\begin{equation*}
\Gamma \left( E\right) \overset{\nabla ^{E}}{\rightarrow }\Gamma \left(
T^{\ast }M\otimes E\right) \overset{\cong }{\rightarrow }\Gamma \left(
TM\otimes E\right) \overset{\pi }{\rightarrow }\Gamma \left( Q\otimes
E\right) \overset{c}{\rightarrow }\Gamma \left( E\right) .
\end{equation*}%
We now calculate the formal adjoint of $A_{Q}$, in precisely the same way
that we showed the formal self-adjointness of the ordinary Dirac operator.
Letting $\left( s_{1},s_{2}\right) $ denote the pointwise inner product of
sections of $E$, we have%
\begin{eqnarray*}
\left( A_{Q}s_{1},s_{2}\right) &=&\sum_{j=1}^{q}\left( c\left( \pi
f_{j}\right) \nabla _{f_{j}}^{E}s_{1},s_{2}\right) \\
&=&\sum_{j=1}^{q}-\left( \nabla _{f_{j}}^{E}s_{1},c\left( \pi f_{j}\right)
s_{2}\right) .
\end{eqnarray*}%
Since$\nabla ^{E}$ is a metric connection, 
\begin{eqnarray}
\left( A_{Q}s_{1},s_{2}\right) &=&\sum \left( -f_{j}\left( s_{1},c\left( \pi
f_{j}\right) s_{2}\right) +\left( s_{1},\nabla _{f_{j}}^{E}c\left( \pi
f_{j}\right) s_{2}\right) \right)  \notag \\
&=&\sum {\Bigg (}-f_{j}\left( s_{1},c\left( \pi f_{j}\right) s_{2}\right)
+\left( s_{1},c\left( \pi f_{j}\right) \nabla _{f_{j}}^{E}s_{2}\right) 
\notag \\
&&+\left( s_{1},c\left( \pi \nabla _{f_{j}}^{M}\pi f_{j}\right) s_{2}\right) 
{\Bigg ),}  \label{AQexpression}
\end{eqnarray}%
by the $\mathbb{C}\mathrm{l}\left( Q\right) $-compatibility. Now, we do not
have the freedom to choose the frame so that the covariant derivatives
vanish at a certain point, because we know nothing about the distribution $Q$%
. Hence we define the vector fields 
\begin{equation*}
V=\sum_{j=1}^{q}\pi \nabla _{f_{j}}^{M}f_{j}~~,~H^{L}=\sum_{j=q+1}^{n}\pi
\nabla _{f_{j}}^{M}f_{j}~.
\end{equation*}%
Note that $H^{L}$ is precisely the mean curvature of the distribution $%
L=Q^{\bot }$. Further, letting $\omega $ be the one-form defined by 
\begin{equation*}
\omega \left( X\right) =\left( s_{1},c\left( \pi X\right) s_{2}\right) ,
\end{equation*}%
and letting $\left( f_{1},...,f_{q},f_{q+1},...,f_{n}\right) $ be an
extension of the frame of $Q$ to be an orthonormal frame of $TM$,%
\begin{eqnarray*}
\delta \omega &=&-\sum_{j=1}^{n}i\left( f_{j}\right) \nabla _{f_{j}}\omega \\
&=&-\sum_{j=1}^{n}\left( f_{j}\omega \left( f_{j}\right) -\omega \left(
\nabla _{f_{j}}f_{j}\right) \right) \\
&=&\sum_{j=1}^{n}\left( -f_{j}\left( s_{1},c\left( \pi f_{j}\right)
s_{2}\right) +\left( s_{1},c\left( \pi \nabla _{f_{j}}^{M}f_{j}\right)
s_{2}\right) \right) \\
&=&\left( s_{1},c\left( V+H^{L}\right) s_{2}\right) +\sum_{j=1}^{n}\left(
-f_{j}\left( s_{1},c\left( \pi f_{j}\right) s_{2}\right) \right) \\
&=&\left( s_{1},c\left( V+H^{L}\right) s_{2}\right) +\sum_{j=1}^{q}\left(
-f_{j}\left( s_{1},c\left( \pi f_{j}\right) s_{2}\right) \right) .
\end{eqnarray*}%
From (\ref{AQexpression}) we have 
\begin{eqnarray*}
\left( A_{Q}s_{1},s_{2}\right) &=&\delta \omega -\left( s_{1},c\left(
V+H^{L}\right) s_{2}\right) \\
&&+\left( s_{1},A_{Q}s_{2}\right) +\left( s_{1},c\left( V\right) s_{2}\right)
\\
\left( A_{Q}s_{1},s_{2}\right) &=&\delta \omega +\left(
s_{1},A_{Q}s_{2}\right) -\left( s_{1},c\left( H^{L}\right) s_{2}\right) .
\end{eqnarray*}%
Thus, by integrating over the manifold (which sends $\delta \omega $ to
zero), we see that the formal $L^{2}$-adjoint of $A_{Q}$ is%
\begin{equation*}
A_{Q}^{\ast }=A_{Q}-c\left( H^{L}\right) .
\end{equation*}%
Since $c\left( H^{L}\right) $ is skew-adjoint, the new operator%
\begin{equation}
D_{Q}=A_{Q}-\frac{1}{2}c\left( H^{L}\right)  \label{DQdef}
\end{equation}%
is formally self-adjoint.

A quick look at \cite{C} yields the following.

\begin{theorem}
(in \cite{PrRi}) For each distribution $Q\subset TM$ and every bundle $E$ of 
$\mathbb{C}\mathrm{l}\left( Q\right) $-modules, the transversally elliptic
operator $D_{Q}$ defined by (\ref{AQdef}) and (\ref{DQdef}) is essentially
self-adjoint.
\end{theorem}

\begin{remark}
In general, the spectrum of $D_{Q}$ is not necessarily discrete. In the case
of Riemannian foliations, we identify $Q$ with the normal bundle of the
foliation, and one typically restricts to the space of basic sections. In
this case, the spectrum of $D_{Q}$ restricted to the basic sections is
discrete.
\end{remark}

\begin{exercise}
Let $M=T^{2}=\mathbb{R}^{2}\diagup \mathbb{Z}^{2}$, and consider the
distribution $Q$ defined by the vectors parallel to $\left( 1,r\right) $
with $r\in \mathbb{R}$. Calculate the operator $D_{Q}$ and its spectrum,
where the Clifford multiplication is just complex number multiplication (on
a trivial bundle $E_{x}=\mathbb{C}$). Does it make a difference if $r$ is
rational?
\end{exercise}

\begin{exercise}
With $M$ and $Q$ as in the last exercise, let $E$ be the bundle $\wedge
^{\ast }Q^{\ast }$. Now calculate $D_{Q}$ and its spectrum.
\end{exercise}

\begin{exercise}
Consider the radially symmetric Heisenberg distribution, defined as follows.
Let $\alpha \in \Omega ^{1}\left( \mathbb{R}^{3}\right) $ be the
differential form 
\begin{eqnarray*}
\alpha &=&dz-\frac{1}{2}r^{2}d\theta \\
&=&dz-\frac{1}{2}\left( xdy-ydx\right) .
\end{eqnarray*}%
Note that 
\begin{equation*}
d\alpha =-rdr\wedge d\theta =-dx\wedge dy,
\end{equation*}%
so that $\alpha $ is a contact form because%
\begin{equation*}
\alpha \wedge d\alpha =-dx\wedge dy\wedge dz\neq 0
\end{equation*}%
at each point of $H$. The two-dimensional distribution $Q\subset \mathbb{R}%
^{3}$ is defined as $Q=\ker \alpha $. Calculate the operator $D_{Q}$.
\end{exercise}

\begin{exercise}
Let $\left( M,\alpha \right) $ be a manifold of dimension $2n+1$ with
contact form $\alpha $; that is, $\alpha $ is a one-form such that 
\begin{equation*}
\alpha \wedge \left( d\alpha \right) ^{n}
\end{equation*}%
is everywhere nonsingular. The distribution $Q=\ker \alpha $ is the contact
distribution. Calculate the mean curvature of $Q$ in terms of $\alpha $.
\end{exercise}

\begin{exercise}
(This example is in the paper \cite{PrRi}.)We consider the torus $M=\left( 
\mathbb{R}\diagup 2\pi \mathbb{Z}\right) ^{2}$ with the metric $e^{2g\left(
y\right) }dx^{2}+dy^{2}$ for some $2\pi $-periodic smooth function $g$.
Consider the orthogonal distributions $L=\mathrm{span}\left\{ \partial
_{y}\right\} $ and $Q=\mathrm{span}\left\{ \partial _{x}\right\} $. Let $E$
be the trivial complex line bundle over $M$, and let $\mathbb{C}\mathrm{l}%
\left( Q\right) $ and $\mathbb{C}\mathrm{l}\left( L\right) $ both act on $E$
via $c\left( \partial _{y}\right) =i=c\left( e^{-g\left( y\right) }\partial
_{x}\right) $. Show that the mean curvatures of these distributions are%
\begin{equation*}
H^{Q}=-g^{\prime }\left( y\right) \partial _{y}\text{ and }H^{L}=0
\end{equation*}%
From formulas (\ref{AQdef}) and (\ref{DQdef}),%
\begin{equation*}
A_{L}=i\partial _{y},\text{ and }D_{L}=i\left( \partial _{y}+\frac{1}{2}%
g^{\prime }\left( y\right) \right) .
\end{equation*}%
Show that the spectrum $\sigma \left( D_{L}\right) =\mathbb{Z}$ is a set
consisting of eigenvalues of infinite multiplicity.
\end{exercise}

\begin{exercise}
In the last example, show that the operator%
\begin{equation*}
D_{Q}=ie^{-g\left( y\right) }\partial _{x}
\end{equation*}%
has spectrum%
\begin{equation*}
\sigma \left( D_{Q}\right) =\bigcup\limits_{n\in \mathbb{Z}}n\left[ a,b%
\right] ,
\end{equation*}%
where $\left[ a,b\right] $ is the range of $e^{-g\left( y\right) }$.
\end{exercise}

\section{Basic Dirac operators on Riemannian foliations}

The results of this section are joint work with G. Habib and can be found in 
\cite{HabRi} and \cite{HabRi2}.

\subsection{Invariance of the spectrum of basic Dirac operators\label%
{basicDiracSection}}

Suppose a closed manifold $M$ is endowed with the structure of a Riemannian
foliation $\left( M,\mathcal{F~},g_{Q}~\right) $. The word \textbf{Riemannian%
} means that there is a metric on the local space of leaves --- a
holonomy-invariant transverse metric $g_{Q}$ on the normal bundle $%
Q=TM\diagup T\mathcal{F~}$ . The phrase \textbf{holonomy-invariant} means
the transverse Lie derivative $\mathcal{L}_{X}g_{Q}$ is zero for all
leafwise vector fields $X\in \Gamma (T\mathcal{F})$.

We often assume that the manifold is endowed with the additional structure
of a \textbf{bundle-like metric} \cite{Re}, i.e. the metric $g$ on $M$
induces the metric on $Q\simeq N\mathcal{F~}=\left( T\mathcal{F}\right)
^{\perp }$. Every Riemannian foliation admits bundle-like metrics that are
compatible with a given $\left( M,\mathcal{F},g_{Q}\right) $ structure.
There are many choices, since one may freely choose the metric along the
leaves and also the transverse subbundle $N\mathcal{F}$. We note that a
bundle-like metric on a smooth foliation is exactly a metric on the manifold
such that the leaves of the foliation are locally equidistant. There are
topological restrictions to the existence of bundle-like metrics (and thus
Riemannian foliations). Important examples of requirements for the existence
of a Riemannian foliations may be found in \cite{KT1}, \cite{KT2}, \cite{Mo}%
, \cite{To}, \cite{Wo}, \cite{Tar}. One geometric requirement is that, for
any metric on the manifold, the orthogonal projection 
\begin{equation*}
P:L^{2}\left( \Omega \left( M\right) \right) \rightarrow L^{2}\left( \Omega
\left( M,\mathcal{F}\right) \right) 
\end{equation*}%
must map the subspace of smooth forms onto the subspace of smooth basic
forms (\cite{PaRi}). Recall that \textbf{basic forms} are forms that depend
only on the transverse variables. The space $\Omega \left( M,\mathcal{F}%
\right) $ of basic forms is defined invariantly as%
\begin{equation*}
\Omega \left( M,\mathcal{F}\right) =\left\{ \beta \in \Omega \left( M\right)
:i\left( X\right) \beta =0\text{ and }i\left( X\right) d\beta =0~\text{for
all }X\in \Gamma \left( T\mathcal{F}\right) \right\} .
\end{equation*}%
The basic forms $\Omega \left( M,\mathcal{F}\right) $ are preserved by the
exterior derivative, and the resulting cohomology is called \textbf{basic
cohomology} $H^{\ast }\left( M,\mathcal{F}\right) $ (see also Section ). It
is known that the basic cohomology groups are finite-dimensional in the
Riemannian foliation case. See \cite{EK}, \cite{ElKacNicol}, \cite{KT1}, 
\cite{KTFol}, \cite{KT4}, \cite{Gh} for facts about basic cohomology and
Riemannian foliations. For later use, the \textbf{basic Euler characteristic}
is defined to be 
\begin{equation*}
\chi \left( M,\mathcal{F}\right) =\sum \left( -1\right) ^{j}\dim H^{j}\left(
M,\mathcal{F}\right) .
\end{equation*}

We now discuss the construction of the basic Dirac operator, a construction
which requires a choice of bundle-like metric. See \cite{KTFol}, \cite{KT4}, 
\cite{KT1}, \cite{DGKY}, \cite{GlK}, \cite{PrRi}, \cite{Ju}, \cite{JuRi}, 
\cite{Hab}, \cite{HabRi}, \cite{BrKRi1}, \cite{BKRi2} for related results.
Let $(M,\mathcal{F})$ be a Riemannian manifold endowed with a Riemannian
foliation. Let $E\rightarrow M$ be a foliated vector bundle (see \cite{KT2})
that is a bundle of $\mathbb{C}\mathrm{l}(Q)$ Clifford modules with
compatible connection $\nabla ^{E}$. This means that foliation lifts to a
horizontal foliation in $TE$. Another way of saying this is that connection
is flat along the leaves of $\mathcal{F}$. When this happens, it is always
possible to choose a basic connection for $E$ --- that is, a connection for
which the connection and curvature forms are actually (Lie algebra-valued)
basic forms.

Let $A_{N\mathcal{F}}$ and $D_{N\mathcal{F}}$ be the associated transversal
Dirac operators as in the previous section. The transversal Dirac operator $%
A_{N\mathcal{F}}$ fixes the basic sections $\Gamma _{b}(E)\subset \Gamma (E)$
(i.e. $\Gamma _{b}(E)=\{s\in \Gamma (E):\nabla _{X}^{E}s=0$ for all $X\in
\Gamma (T\mathcal{F})\}$) but is not symmetric on this subspace. Let $%
P_{b}:L^{2}\left( \Gamma \left( E\right) \right) \rightarrow L^{2}\left(
\Gamma _{b}\left( E\right) \right) $ be the orthogonal projection, which can
be shown to map smooth sections to smooth basic sections. We define the
basic Dirac operator to be%
\begin{eqnarray*}
D_{b} &:&=P_{b}D_{N\mathcal{F}}P_{b} \\
&=&A_{N\mathcal{F}}-\frac{1}{2}c\left( \kappa _{b}^{\sharp }\right) :\Gamma
_{b}\left( E\right) \rightarrow \Gamma _{b}\left( E\right)
\end{eqnarray*}%
\begin{equation*}
P_{b}A_{N\mathcal{F}}P_{b}=A_{N\mathcal{F}}P_{b},~P_{b}c\left( \kappa
^{\sharp }\right) P_{b}=c\left( \kappa _{b}^{\sharp }\right) P_{b}
\end{equation*}%
Here, $\kappa _{b}$ is the $L^{2}$-orthogonal projection of $\kappa $ onto
the space of basic forms as explained above, and $\kappa _{b}^{\sharp }$ is
the corresponding basic vector field. Then $D_{b}$ is an essentially
self-adjoint, transversally elliptic operator on $\Gamma _{b}(E)$. The local
formula for $D_{b}$ is 
\begin{equation*}
D_{b}s=\sum_{i=1}^{q}e_{i}\cdot \nabla _{e_{i}}^{E}s-\frac{1}{2}\kappa
_{b}^{\sharp }\cdot s~,
\end{equation*}%
where $\{e_{i}\}_{i=1,\cdots ,q}$ is a local orthonormal frame of $Q$. Then $%
D_{b}$ has discrete spectrum (\cite{GlK}, \cite{DGKY}, \cite{BKRi2}).

An example of the basic Dirac operator is as follows. Using the bundle $%
\wedge ^{\ast }Q^{\ast }$ as the Clifford bundle with Clifford action $%
e\cdot ~=e^{\ast }\wedge -e^{\ast }\lrcorner $ in analogy to the ordinary de
Rham operator, we have%
\begin{eqnarray*}
D_{b} &=&d+\delta _{b}-\frac{1}{2}\kappa _{b}\lrcorner -\frac{1}{2}\kappa
_{b}\wedge . \\
&=&\widetilde{d}+\widetilde{\delta }.
\end{eqnarray*}%
One might have incorrectly guessed that $d+\delta _{b}$ is the basic de Rham
operator in analogy to the ordinary de Rham operator, for this operator is
essentially self-adjoint, and the associated basic Laplacian yields basic
Hodge theory that can be used to compute the basic cohomology. The square $%
D_{b}^{2}$ of this operator and the basic Laplacian $\Delta _{b}$ do have
the same principal transverse symbol. In \cite{HabRi}, we showed the
invariance of the spectrum of $D_{b}$ with respect to a change of metric on $%
M$ in any way that keeps the transverse metric on the normal bundle intact
(this includes modifying the subbundle $N\mathcal{F}\subset TM$, as one must
do in order to make the mean curvature basic, for example). That is,

\begin{theorem}
\label{inv}(In \cite{HabRi} ) Let $(M,\mathcal{F})$ be a compact Riemannian
manifold endowed with a Riemannian foliation and basic Clifford bundle $%
E\rightarrow M$. The spectrum of the basic Dirac operator is the same for
every possible choice of bundle-like metric that is associated to the
transverse metric on the quotient bundle $Q$.
\end{theorem}

We emphasize that the basic Dirac operator $D_{b}$ depends on the choice of
bundle-like metric, not merely on the Clifford structure and Riemannian
foliation structure, since both projections $T^{\ast }M\rightarrow Q^{\ast }$
and $P$ depend on the leafwise metric. It is well-known that the eigenvalues
of the basic Laplacian $\Delta _{b}$ (closely related to $D_{b}^{2}$) depend
on the choice of bundle-like metric; for example, in \cite[Corollary 3.8]%
{Ri2}, it is shown that the spectrum of the basic Laplacian on functions
determines the $L^{2}$-norm of the mean curvature on a transversally
oriented foliation of codimension one. If the foliation were taut, then a
bundle-like metric could be chosen so that the mean curvature is identically
zero, and other metrics could be chosen where the mean curvature is nonzero.
This is one reason why the invariance of the spectrum of the basic Dirac
operator is a surprise.

\begin{exercise}
Suppose that $\mathcal{S}$ is a closed subspace of a Hilbert space $\mathcal{%
H}$, and let $L:\mathcal{H}\rightarrow \mathcal{H}$ be a bounded linear map
such that $L\left( \mathcal{S}\right) \subseteq \mathcal{S}$. Let $L_{%
\mathcal{S}}$ denote the restriction $L_{\mathcal{S}}:\mathcal{S}\rightarrow 
\mathcal{S}$ defined by $L_{\mathcal{S}}\left( v\right) =L\left( v\right) $
for all $v\in \mathcal{S}$. Prove that the adjoint of $L_{\mathcal{S}}$
satisfies $L_{\mathcal{S}}^{\ast }\left( v\right) =P_{\mathcal{S}}L^{\ast
}\left( v\right) $, where $L^{\ast }$ is the adjoint of $L$ and $P_{\mathcal{%
S}}$ is the orthogonal projection of $\mathcal{H}$ onto $\mathcal{S}$. Show
that the maximal subspace $\mathcal{W}\subseteq \mathcal{S}$ such that $%
\left. L_{\mathcal{S}}^{\ast }\right\vert _{\mathcal{W}}=\left. L^{\ast
}\right\vert _{\mathcal{W}}$ satisfies%
\begin{equation*}
\mathcal{W}=\left( \mathcal{S}\cap L\left( \mathcal{S}^{\bot }\right)
\right) ^{\bot _{\mathcal{S}}},
\end{equation*}%
where $\mathcal{S}^{\bot }$ is the orthogonal complement of $\mathcal{S}$ in 
$\mathcal{H}$ and the superscript $\bot _{\mathcal{S}}$ denotes the
orthogonal complement in $\mathcal{S}$.
\end{exercise}

\begin{exercise}
Prove that the metric on a Riemannian manifold $M$ with a smooth foliation $%
\mathcal{F}$ is bundle-like if and only if the normal bundle $N\mathcal{F}$
with respect to that metric is totally geodesic.
\end{exercise}

\begin{exercise}
Let $\left( M,\mathcal{F}\right) $ be a transversally oriented Riemannian
foliation of codimension $q$ with bundle-like metric, and let $\nu $ be the
transversal volume form. The \textbf{transversal Hodge star operator} $%
\overline{\ast }:$ $\wedge ^{\ast }Q^{\ast }\rightarrow \wedge ^{\ast
}Q^{\ast }$ is defined by%
\begin{equation*}
\alpha \wedge \overline{\ast }\beta =\left( \alpha ,\beta \right) \nu ,
\end{equation*}%
for $\alpha ,\beta \in \Omega ^{k}\left( M,\mathcal{F}\right) $, so that $%
\overline{\ast }1=\nu $, $\overline{\ast }\nu =1$. Let the transversal
codifferential $\delta _{T}:$ be defined by%
\begin{equation*}
\delta _{T}=\left( -1\right) ^{qk+q+1}\overline{\ast }d\overline{\ast }%
:\Omega ^{k}\left( M,\mathcal{F}\right) \rightarrow \Omega ^{k-1}\left( M,%
\mathcal{F}\right) .
\end{equation*}%
As above, let $\delta _{b}$ be the adjoint of $d$ with respect to $%
L^{2}\left( \Omega \left( M,\mathcal{F}\right) \right) $.\newline
Prove the following identities:

\begin{itemize}
\item $\overline{\ast }^{2}=\left( -1\right) ^{k\left( q-k\right) }$ on
basic $k$-forms.

\item If $\beta $ is a basic one-form, then $\left( \beta \lrcorner \right)
=\left( -1\right) ^{q\left( k+1\right) }\overline{\ast }\left( \beta \wedge
\right) \overline{\ast }$ as operators on basic $k$-forms.

\item $\delta _{b}=\delta _{T}+\kappa _{b}\lrcorner $

\item $\delta _{b}\nu =\overline{\ast }\kappa _{b}$

\item $d\kappa _{b}=0$ (Hint: compute $\delta _{b}^{2}\nu $.)

\item $\overline{\ast }\widetilde{d}=\pm \widetilde{\delta }\overline{\ast }$%
, with $\widetilde{d}=d-\frac{1}{2}\kappa _{b}\wedge ,\,\,\widetilde{\delta }%
=\delta _{b}-\frac{1}{2}\kappa _{b}\lrcorner .$
\end{itemize}
\end{exercise}

\subsection{The basic de Rham operator}

From the previous section, the basic de Rham operator is $D_{b}=\widetilde{d}%
+\widetilde{\delta }$ acting on basic forms, where%
\begin{equation*}
\widetilde{d}=d-\frac{1}{2}\kappa _{b}\wedge ,\,\,\widetilde{\delta }=\delta
_{b}-\frac{1}{2}\kappa _{b}\lrcorner .
\end{equation*}%
Unlike the ordinary and well-studied basic Laplacian, the eigenvalues of $%
\widetilde{\Delta }=D_{b}^{2}$ are invariants of the Riemannian foliation
structure alone and independent of the choice of compatible bundle-like
metric. The operators $\widetilde{d}$ and $\widetilde{\delta }$ have
following interesting properties.

\begin{lemma}
$\widetilde{\delta }$ is the formal adjoint of $\widetilde{d}$ .
\end{lemma}

\begin{lemma}
The maps $\widetilde{d}$ and $\widetilde{\delta }$ are differentials; that
is, $\widetilde{d}^{2}=0$, $\widetilde{\delta }^{2}=0$. As a result, $%
\widetilde{d}$ and $\widetilde{\delta }$ commute with $\widetilde{\Delta }%
=D_{b}^{2}$, and $\ker \left( \widetilde{d}+\widetilde{\delta }\right) =\ker
\left( \widetilde{\Delta }\right) $.
\end{lemma}

Let $\Omega^{k}\left( M,\mathcal{F}\right) $ denote the space of basic $k$%
-forms (either set of smooth forms or $L^{2}$-completion thereof), let $%
\widetilde{d}^{k}$ and $\widetilde{\delta _{b}}^{k}$ be the restrictions of $%
\widetilde{d}$ and $\widetilde{\delta _{b}}$ to $k$-forms, and let $%
\widetilde{\Delta }^{k}$ denote the restriction of $D_{b}^{2}$ to basic $k$%
-forms.

\begin{proposition}
(Hodge decomposition)\label{Hodge Theorem} We have 
\[
\Omega^{k}\left( M,%
\mathcal{F}\right) =\mathrm{image}\left( \widetilde{d}^{k-1}\right) \oplus 
\mathrm{image}\left( \widetilde{\delta _{b}}^{k+1}\right) \oplus \ker \left( 
\widetilde{\Delta }^{k}\right) ,
\] 
an $L^{2}$-orthogonal direct sum. Also, $%
\ker \left( \widetilde{\Delta }^{k}\right) $ is finite-dimensional and
consists of smooth forms.
\end{proposition}

We call $\ker \left( \widetilde{\Delta }\right) $ the space $\widetilde{%
\Delta }$-harmonic forms. In the remainder of this section, we assume that
the foliation is transversally oriented so that the transversal Hodge $%
\overline{\ast }$ operator is well-defined.

\begin{definition}
We define the basic $\widetilde{d}$-cohomology $\widetilde{H}^{\ast }\left(
M,\mathcal{F}\right) $ by 
\begin{equation*}
\widetilde{H}^{k}\left( M,\mathcal{F}\right) =\frac{\ker \widetilde{d}^{k}}{%
\mathrm{image}~\widetilde{d}^{k-1}}.
\end{equation*}
\end{definition}

The following proposition follows from standard arguments and the Hodge
theorem (Theorem \ref{Hodge Theorem}).

\begin{proposition}
The finite-dimensional vector spaces $\widetilde{H}^{k}\left( M,\mathcal{F}%
\right) $ and $\ker $ $\widetilde{\Delta }^{k}=\ker \left( \widetilde{d}+%
\widetilde{\delta }\right) ^{k}$ are naturally isomorphic.
\end{proposition}

We observe that for every choice of bundle-like metric, the differential $%
\widetilde{d}=d-\frac{1}{2}\kappa _{b}\wedge $ changes, and thus the
cohomology groups change. However, note that $\kappa _{b}$ is the only part
that changes; for any two bundle-like metrics $g_{M}$, $g_{M}^{\prime }$ and
associated $\kappa _{b}$, $\kappa _{b}^{\prime }$ compatible with $\left( M,%
\mathcal{F},g_{Q}\right) $, we have $\kappa _{b}^{\prime }=\kappa _{b}+dh$
for some basic function $h$ (see \cite{Al}). In the proof of the main
theorem in \cite{HabRi}, we essentially showed that the the basic de Rham
operator $D_{b}$ is then transformed by $D_{b}^{\prime
}=e^{h/2}D_{b}e^{-h/2} $. Applying this to our situation, we see that the $%
\left( \ker D_{b}^{\prime }\right) =e^{h/2}\ker D_{b}$, and thus the
cohomology groups are the same dimensions, independent of choices. To see
this in our specific situation, note that if $\alpha \in \Omega^{k}\left( M,%
\mathcal{F}\right) $ satisfies $\widetilde{d}\alpha =0$, then%
\begin{eqnarray*}
\left( \widetilde{d}\right) ^{\prime }\left( e^{h/2}\alpha \right) &=&\left(
d-\frac{1}{2}\kappa _{b}\wedge -\frac{1}{2}dh\wedge \right) \left(
e^{h/2}\alpha \right) \\
&=&e^{h/2}d\alpha +\frac{1}{2}e^{h/2}dh\wedge \alpha -\frac{e^{h/2}}{2}%
\kappa _{b}\wedge \alpha -\frac{e^{h/2}}{2}dh\wedge \alpha \\
&=&e^{h/2}d\alpha -\frac{e^{h/2}}{2}\kappa _{b}\wedge \alpha =e^{h/2}\left(
d-\frac{1}{2}\kappa _{b}\wedge \right) \alpha =e^{h/2}\widetilde{d}\alpha =0.
\end{eqnarray*}%
Similarly, as in \cite{HabRi} one may show $\ker \left( \widetilde{\delta }%
\right) ^{\prime }=e^{h/2}\ker \left( \widetilde{\delta }\right) $, through
a slightly more difficult computation. Thus, we have

\begin{theorem}
(Conformal invariance of cohomology groups) Given a Riemannian foliation $%
\left( M,\mathcal{F},g_{Q}\right) $ and two bundle-like metrics $g_{M}$ and $%
g_{M}^{\prime }$ compatible with $g_{Q}$, the $\widetilde{d}$-cohomology
groups $\widetilde{H}^{k}\left( M,\mathcal{F}\right) $ are isomorphic, and
that isomorphism is implemented by multiplication by a positive basic
function. Further, the eigenvalues of the corresponding basic de Rham
operators $D_{b}$ and $D_{b}^{\prime }$ are identical, and the eigenspaces
are isomorphic via multiplication by that same positive function.

\begin{corollary}
The dimensions of $\widetilde{H}^{k}\left( M,\mathcal{F}\right) $ and the
eigenvalues of $D_{b}$ (and thus of $\widetilde{\Delta }=D_{b}^{2}$) are
invariants of the Riemannian foliation structure $\left( M,\mathcal{F}%
,g_{Q}\right) $, independent of choice of compatible bundle-like metric $%
g_{M}$.
\end{corollary}
\end{theorem}

\begin{exercise}
Show that if $\alpha $ is any closed form, then $\left( d+\alpha \wedge
\right) ^{2}=0$.
\end{exercise}

\begin{exercise}
Show that $\widetilde{\delta }$ is the formal adjoint of $\widetilde{d}$.
\end{exercise}

\begin{exercise}
Show that a Riemannian foliation $\left( M,\mathcal{F}\right) $ is taut if
and only if $\widetilde{H}^{0}\left( M,\mathcal{F}\right) $ is nonzero. (A
Riemannian foliation is \textbf{taut} if there exists a bundle-like metric
for which the leaves are minimal submanifolds and thus have zero mean
curvature.)
\end{exercise}

\begin{exercise}
(This example is contained in \cite{HabRi2}.) \label{CarExampleExercise}This
Riemannian foliation is the famous Carri\`{e}re example from \cite{Car} in
the $3$-dimensional case. Let $A$ be a matrix in $\mathrm{SL}_{2}(\mathbb{Z}%
) $ of trace strictly greater than $2$. We denote respectively by $V_{1}$
and $V_{2}$ the eigenvectors associated with the eigenvalues $\lambda $ and $%
\frac{1}{\lambda }$ of $A$ with $\lambda >1$ irrational. Let the hyperbolic
torus $\mathbb{T}_{A}^{3}$ be the quotient of $\mathbb{T}^{2}\times \mathbb{R%
}$ by the equivalence relation which identifies $(m,t)$ to $(A(m),t+1)$. The
flow generated by the vector field $V_{2}$ is a transversally Lie foliation
of the affine group. We denote by $K$ the holonomy subgroup. The affine
group is the Lie group $\mathbb{R}^{2}$ with multiplication $%
(t,s).(t^{\prime },s^{\prime })=(t+t^{\prime },\lambda ^{t}s^{\prime }+s)$,
and the subgroup $K$ is 
\begin{equation*}
K=\{(n,s),n\in \mathbb{Z},s\in \mathbb{R}\}.
\end{equation*}%
We choose the bundle-like metric (letting $\left( x,s,t\right) $ denote the
local coordinates in the $V_{2}$ direction, $V_{1}$ direction, and $\mathbb{R%
}$ direction, respectively) as 
\begin{equation*}
g=\lambda ^{-2t}dx^{2}+\lambda ^{2t}ds^{2}+dt^{2}.
\end{equation*}%
Prove that:

\begin{itemize}
\item The mean curvature of the flow is $\kappa =\kappa _{b}=\log \left(
\lambda \right) dt$.

\item The twisted basic cohomology groups are all trivial.

\item The ordinary basic cohomology groups satisfy $H^{0}\left( M,\mathcal{F}%
\right) \cong \mathbb{R}$ , $H^{1}\left( M,\mathcal{F}\right) \cong \mathbb{R%
}$ , $H^{2}\left( M,\mathcal{F}\right) \cong \left\{ 0\right\} $.

\item The flow is not taut.
\end{itemize}
\end{exercise}

\subsection{Poincar\'{e} duality and consequences}

\begin{theorem}
(Poincar\'{e} duality for $\widetilde{d}$-cohomology) Suppose that the
Riemannian foliation $\left( M,\mathcal{F},g_{Q}\right) $ is transversally
oriented and is endowed with a bundle-like metric. For each $k$ such that $%
0\leq k\leq q$ and any compatible choice of bundle-like metric, the map $%
\overline{\ast }:\Omega^{k}\left( M,\mathcal{F}\right) \rightarrow
\Omega^{q-k}\left( M,\mathcal{F}\right) $ induces an isomorphism on the $%
\widetilde{d}$-cohomology. Moreover, $\overline{\ast }$ maps the $\ker 
\widetilde{\Delta }^{k}$ isomorphically onto $\ker \widetilde{\Delta }^{q-k}$%
, and it maps the $\lambda $-eigenspace of $\widetilde{\Delta }^{k}$
isomorphically onto the $\lambda $-eigenspace of $\widetilde{\Delta }^{q-k}$%
, for all $\lambda \geq 0$.
\end{theorem}

This resolves the problem of the failure of Poincar\'{e} duality to hold for
standard basic cohomology (see \cite{KTduality}, \cite{To}).

\begin{corollary}
\label{OddCodimZeroEuler}Let $\left( M,\mathcal{F}\right) $ be a smooth
transversally oriented foliation of odd codimension that admits a transverse
Riemannian structure. Then the Euler characteristic associated to the $%
\widetilde{H}^{\ast }\left( M,\mathcal{F}\right) $ vanishes.
\end{corollary}

The following fact is a new result for ordinary basic cohomology of
Riemannian foliations. Ordinary basic cohomology does not satisfy Poincar%
\'{e} duality; in fact, the top-dimensional basic cohomology group is zero
if and only if the foliation is not taut. Also, leaf closures of a
transversally oriented foliation can fail to be transversally oriented, so
orientation is also a tricky issue.

\begin{corollary}
Let $\left( M,\mathcal{F}\right) $ be a smooth transversally oriented
foliation of odd codimension that admits a transverse Riemannian structure.
Then the Euler characteristic associated to the ordinary basic cohomology $%
H^{\ast }\left( M,\mathcal{F}\right) $ vanishes.
\end{corollary}

\begin{proof}
The basic Euler characteristic is the basic index of the operator $%
D_{0}=d+\delta _{B}:\Omega ^{\mathrm{even}}\left( M,\mathcal{F}\right)
\rightarrow \Omega ^{\mathrm{odd}}\left( M,\mathcal{F}\right) $. See \cite%
{BrKRi1}, \cite{DGKY}, \cite{BePaRi}, \cite{EK} for information on the basic
index and basic Euler characteristic. The crucial property for us is that
the basic index of $D_{0}$ is a Fredholm index and is invariant under
perturbations of the operator through transversally elliptic operators that
map the basic forms to themselves. In particular, the family of operators $%
D_{t}=d+\delta _{b}-\frac{t}{2}\kappa _{b}\lrcorner -\frac{t}{2}\kappa
_{b}\wedge $ for $0\leq t\leq 1$ meets that criteria, and $D_{1}=D_{b}$ is
the basic de Rham operator $D_{b}:\Omega ^{\mathrm{even}}\left( M,\mathcal{F}%
\right) \rightarrow \Omega ^{\mathrm{odd}}\left( M,\mathcal{F}\right) $.
Thus, the basic Euler characteristic of the basic cohomology complex is the
same as the basic Euler characteristic of the $\widetilde{d}$-cohomology
complex. The result follows from the previous corollary.
\end{proof}

\begin{exercise}
Prove that the twisted cohomology class $\left[ \kappa _{b}\right] $ is
always trivial.
\end{exercise}

\begin{exercise}
Prove that if $\left( M,\mathcal{F}\right) $ is not taut, then the ordinary
basic cohomology satisfies $\dim H^{1}\left( M,\mathcal{F}\right) \geq 1$.
\end{exercise}

\begin{exercise}
Prove that there exists a monomorphism from $H^{1}\left( M,\mathcal{F}%
\right) $ to $H^{1}\left( M\right) $.
\end{exercise}

\begin{exercise}
True or false: $\dim H^{2}\left( M,\mathcal{F}\right) \leq \dim H^{2}\left(
M\right) $.
\end{exercise}

\begin{exercise}
Under what conditions is it true that $\dim \widetilde{H}^{1}\left( M,%
\mathcal{F}\right) \geq \dim H^{1}\left( M,\mathcal{F}\right) $ ?
\end{exercise}

\begin{exercise}
(Hard) Find an example of a Riemannian foliation that is not taut and whose
twisted basic cohomology is nontrivial.\newline
(If you give up, find an answer in \cite{HabRi2}.)
\end{exercise}

\section{Natural examples of transversal Dirac operators on $G$-manifolds}

The research content of this section is joint work with I. Prokhorenkov,
from \cite{PrRi}.

\subsection{Equivariant structure of the orthonormal frame bundle\label%
{equivariantStructureSection}}

We first make the important observation that if a Lie group acts effectively
by isometries on a Riemannian manifold, then this action can be lifted to a
free action on the orthonormal frame bundle. Given a complete, connected $G$%
-manifold, the action of $g\in G$ on $M$ induces an action of $dg$ on $TM$,
which in turn induces an action of $G$ on the principal $O\left( n\right) $%
-bundle $F_{O}\overset{p}{\rightarrow }M$ of orthonormal frames over $M$.

\begin{lemma}
\vspace{0in}\label{LiftingLemmaForGroupActions}The action of $G$ on $F_{O}$
is regular, and the isotropy subgroups corresponding to any two points of $%
F_{O}$ are the same.
\end{lemma}

\begin{proof}
Let $H$ be the isotropy subgroup of a frame $f\in F_{O}$. Then $H$ also
fixes $p\left( f\right) \in M$, and since $H$ fixes the frame, its
differentials fix the entire tangent space at $p\left( f\right) $. Since it
fixes the tangent space, every element of $H$ also fixes every frame in $%
p^{-1}\left( p\left( f\right) \right) $; thus every frame in a given fiber
must have the same isotropy subgroup. Since the elements of $H$ map
geodesics to geodesics and preserve distance, a neighborhood of $p\left(
f\right) $ is fixed by $H$. Thus, $H$ is a subgroup of the isotropy subgroup
at each point of that neighborhood. Conversely, if an element of $G$ fixes a
neighborhood of a point $x$ in $M$, then it fixes all frames in $%
p^{-1}\left( x\right) $, and thus all frames in the fibers above that
neighborhood. Since $M$ is connected, we may conclude that every point of $%
F_{O}$ has the same isotropy subgroup $H$, and $H$ is the subgroup of $G$
that fixes every point of $M$.
\end{proof}

\begin{remark}
Since this subgroup $H$ is normal, we often reduce the group $G$ to the
group $G/H$ so that our action is effective, in which case the isotropy
subgroups on $F_{O}$ are all trivial.
\end{remark}

\begin{remark}
A similar idea is also useful in constructing the lifted foliation and the
basic manifold associated to a Riemannian foliation (see \cite{Mo}).
\end{remark}

In any case, the $G$ orbits on $F_{O}$ are diffeomorphic and form a
Riemannian fiber bundle, in the natural metric on $F_{O}$ defined as
follows. The Levi-Civita connection on $M$ determines the horizontal
subbundle $\mathcal{H}$ of $TF_{O}$. We construct the local product metric
on $F_{O}$ using a biinvariant fiber metric and the pullback of the metric
on $M$ to $\mathcal{H}$; with this metric, $F_{O}$ is a compact Riemannian $%
G\times O\left( n\right) $-manifold. The lifted $G$-action commutes with the 
$O\left( n\right) $-action. Let $\mathcal{F}$ denote the foliation of $G$%
-orbits on $F_{O}$, and observe that $F_{O}\overset{\pi }{\rightarrow }%
F_{O}\diagup G=F_{O}\diagup \mathcal{F}$ is a Riemannian submersion of
compact $O\left( n\right) $-manifolds.

Let $E\rightarrow F_{O}$ be a Hermitian vector bundle that is equivariant
with respect to the $G\times O\left( n\right) $ action. Let $\rho
:G\rightarrow U\left( V_{\rho }\right) $ and $\sigma :O\left( n\right)
\rightarrow U\left( W_{\sigma }\right) $ be irreducible unitary
representations. We define the bundle $\mathcal{E}^{\sigma }\rightarrow M$
by 
\begin{equation*}
\mathcal{E}_{x}^{\sigma }=\Gamma \left( p^{-1}\left( x\right) ,E\right)
^{\sigma },
\end{equation*}%
where the superscript $\sigma $ is defined for a $O\left( n\right) $-module $%
Z$ by%
\begin{equation*}
Z^{\sigma }=\mathrm{eval}\left( \mathrm{Hom}_{O\left( n\right) }\left(
W_{\sigma },Z\right) \otimes W_{\sigma }\right) ,
\end{equation*}%
where $\mathbb{\mathrm{eval}}:\mathrm{Hom}_{O\left( n\right) }\left(
W_{\sigma },Z\right) \otimes W_{\sigma }\rightarrow Z$ is the evaluation map 
$\phi \otimes w\mapsto \phi \left( w\right) $. The space $Z^{\sigma }$ is
the vector subspace of $Z$ on which $O\left( n\right) $ acts as a direct sum
of representations of type $\sigma $. The bundle $\mathcal{E}^{\sigma }$ is
a Hermitian $G$-vector bundle of finite rank over $M$. The metric on $%
\mathcal{E}^{\sigma }$ is chosen as follows. For any $v_{x}$,$w_{x}\in 
\mathcal{E}_{x}^{\sigma }$, we define%
\begin{equation*}
\,\left\langle v_{x},w_{x}\right\rangle :=\int_{p^{-1}\left( x\right)
}\left\langle v_{x}\left( y\right) ,w_{x}\left( y\right) \right\rangle
_{y,E}~d\mu _{x}\left( y\right) ,
\end{equation*}%
where $d\mu _{x}$ is the measure on $p^{-1}\left( x\right) $ induced from
the metric on $F_{O}$. See \cite{BrKRi1} for a similar construction.

Similarly, we define the bundle $\mathcal{T}^{\rho }\rightarrow F_{O}\diagup
G$ by%
\begin{equation*}
\mathcal{T}_{y}^{\rho }=\Gamma \left( \pi ^{-1}\left( y\right) ,E\right)
^{\rho },
\end{equation*}%
and $\mathcal{T}^{\rho }\rightarrow F_{O}\diagup G$ is a Hermitian $O\left(
n\right) $-equivariant bundle of finite rank. The metric on $\mathcal{T}%
^{\rho }$ is%
\begin{equation*}
\left\langle v_{z},w_{z}\right\rangle :=\int_{\pi ^{-1}\left( y\right)
}\left\langle v_{z}\left( y\right) ,w_{z}\left( y\right) \right\rangle
_{z,E}~dm_{z}\left( y\right) ,
\end{equation*}%
where $dm_{z}$ is the measure on $\pi ^{-1}\left( z\right) $ induced from
the metric on $F_{O}$.

The vector spaces of sections $\Gamma \left( M,\mathcal{E}^{\sigma }\right) $
and $\Gamma \left( F_{O},E\right) ^{\sigma }$ can be identified via the
isomorphism%
\begin{equation*}
i_{\sigma }:\Gamma \left( M,\mathcal{E}^{\sigma }\right) \rightarrow \Gamma
\left( F_{O},E\right) ^{\sigma },
\end{equation*}%
where for any section $s\in \Gamma \left( M,\mathcal{E}^{\sigma }\right) $, $%
s\left( x\right) \in \Gamma \left( p^{-1}\left( x\right) ,E\right) ^{\sigma
} $ for each $x\in M$, and we let%
\begin{equation*}
i_{\sigma }\left( s\right) \left( f_{x}\right) :=\left. s\left( x\right)
\right\vert _{f_{x}}
\end{equation*}%
for every $f_{x}\in p^{-1}\left( x\right) \subset F_{O}$. Then $i_{\sigma
}^{-1}:\Gamma \left( F_{O},E\right) ^{\sigma }\rightarrow \Gamma \left( M,%
\mathcal{E}^{\sigma }\right) $ is given by%
\begin{equation*}
i_{\sigma }^{-1}\left( u\right) \left( x\right) =\left. u\right\vert
_{p^{-1}\left( x\right) }.
\end{equation*}%
Observe that $i_{\sigma }:\Gamma \left( M,\mathcal{E}^{\sigma }\right)
\rightarrow \Gamma \left( F_{O},E\right) ^{\sigma }$ extends to an $L^{2}$
isometry. Given $u,v\in $ $\Gamma \left( M,\mathcal{E}^{\sigma }\right) $,%
\begin{eqnarray*}
\left\langle u,v\right\rangle _{M} &=&\int_{M}\left\langle
u_{x},v_{x}\right\rangle ~dx=\int_{M}\int_{p^{-1}\left( x\right)
}\left\langle u_{x}\left( y\right) ,v_{x}\left( y\right) \right\rangle
_{y,E}~d\mu _{x}\left( y\right) ~dx \\
&=&\int_{M}\left( \int_{p^{-1}\left( x\right) }\left\langle i_{\sigma
}\left( u\right) ,i_{\sigma }\left( v\right) \right\rangle _{E}~d\mu
_{x}\left( y\right) \right) ~dx \\
&=&\int_{F_{O}}\left\langle i_{\sigma }\left( u\right) ,i_{\sigma }\left(
v\right) \right\rangle _{E}~=\left\langle i_{\sigma }\left( u\right)
,i_{\sigma }\left( v\right) \right\rangle _{F_{O}},
\end{eqnarray*}%
where $dx$ is the Riemannian measure on $M$; we have used the fact that $p$
is a Riemannian submersion. Similarly, we let%
\begin{equation*}
j_{\rho }:\Gamma \left( F_{O}\diagup G,\mathcal{T}^{\rho }\right)
\rightarrow \Gamma \left( F_{O},E\right) ^{\rho }
\end{equation*}%
be the natural identification, which extends to an $L^{2}$ isometry.

Let%
\begin{equation*}
\Gamma \left( M,\mathcal{E}^{\sigma }\right) ^{\alpha }=\mathrm{eval}\left( 
\mathrm{Hom}_{G}\left( V_{\alpha },\Gamma \left( M,\mathcal{E}^{\sigma
}\right) \right) \otimes V_{\alpha }\right) .
\end{equation*}%
Similarly, let%
\begin{equation*}
\Gamma \left( F_{O}\diagup G,\mathcal{T}^{\rho }\right) ^{\beta }=\mathrm{%
eval}\left( \mathrm{Hom}_{G}\left( W_{\beta },\Gamma \left( F_{O}\diagup G,%
\mathcal{T}^{\rho }\right) \right) \otimes W_{\beta }\right) .
\end{equation*}

\begin{theorem}
\label{IsomorphismsOfSectionsTheorem}For any irreducible representations $%
\rho :G\rightarrow U\left( V_{\rho }\right) $ and $\sigma :O\left( n\right)
\rightarrow U\left( W_{\sigma }\right) $, the map $j_{\rho }^{-1}\circ
i_{\sigma }:\Gamma \left( M,\mathcal{E}^{\sigma }\right) ^{\rho }\rightarrow
\Gamma \left( F_{O}\diagup G,\mathcal{T}^{\rho }\right) ^{\sigma }$ is an
isomorphism (with inverse $i_{\sigma }^{-1}\circ j_{\rho }$) that extends to
an $L^{2}$-isometry.
\end{theorem}

\begin{exercise}
Prove that if $M$ is a Riemannian manifold, then the orthonormal frame
bundle of $M$ has trivial tangent bundle.
\end{exercise}

\begin{exercise}
Suppose that a compact Lie group acts smoothly on a Riemannian manifold.
Prove that there exists a metric on the manifold such that the Lie group
acts isometrically.
\end{exercise}

\begin{exercise}
\label{Z2actiononTorusExercise}Let $\mathbb{Z}_{2}$ act on $T^{2}=$ $\mathbb{%
R}^{2}\diagup \mathbb{Z}^{2}$ with an action generated by $\left( x,y\right)
=\left( -x,y\right) $ for $x,y\in \mathbb{R}\diagup \mathbb{Z}$.
\end{exercise}

\begin{itemize}
\item Find the quotient space $T^{2}\diagup \mathbb{Z}_{2}$.

\item Find all the irreducible representations of $\mathbb{Z}_{2}$. (Hint:
they are all homomorphisms $\rho :\mathbb{Z}_{2}\rightarrow U\left( 1\right) 
$.)

\item Find the orthonormal frame bundle $F_{O}$, and determine the induced
action of $\mathbb{Z}_{2}$ on $F_{O}$.

\item Find the quotient space $F_{O}\diagup \mathbb{Z}_{2}$, and determine
the induced action of $O\left( 2\right) $ on this manifold.
\end{itemize}

\begin{exercise}
\label{SO3Exercise}Suppose that $M=S^{2}$ is the unit sphere in $\mathbb{R}%
^{3}$. Let $S^{1}$ act on $S^{2}$ by rotations around the $x_{3}$-axis.

\begin{itemize}
\item Show that the oriented orthonormal frame bundle $F_{SO}$ can be
identified with $SO\left( 3\right) $, which in turn can be identified with $%
\mathbb{R}P^{3}$.

\item Show that the lifted $S^{1}$ action on $F_{SO}$ can be realized by the
orbits of a left-invariant vector field on $SO\left( 3\right) $.

\item Find the quotient $F_{SO}\diagup S^{1}$.
\end{itemize}
\end{exercise}

\begin{exercise}
\label{harmonicExercise}Suppose that a compact, connected Lie group acts by
isometries on a Riemannian manifold. Show that all harmonic forms are
invariant under pullbacks by the group action.
\end{exercise}

\subsection{Dirac-type operators on the frame bundle\label{DiracFrameBundle}}

Let $E\rightarrow F_{O}$ be a Hermitian vector bundle of $\mathbb{C}\mathrm{l%
}\left( N\mathcal{F}\right) $ modules that is equivariant with respect to
the $G\times O\left( n\right) $ action. With notation as in previous
sections, we have the transversal Dirac operator $A_{N\mathcal{F}}$ defined
by the composition 
\begin{equation*}
\Gamma \left( F_{O},E\right) \overset{\nabla }{\rightarrow }\Gamma \left(
F_{O},T^{\ast }F_{O}\otimes E\right) \overset{\mathrm{proj}}{\rightarrow }%
\Gamma \left( F_{O},N^{\ast }\mathcal{F}\otimes E\right) \overset{c}{%
\rightarrow }\Gamma \left( F_{O},E\right) .
\end{equation*}%
As explained previously, the operator 
\begin{equation*}
D_{N\mathcal{F}}=A_{N\mathcal{F}}-\frac{1}{2}c\left( H\right)
\end{equation*}%
is a essentially self-adjoint $G\times O\left( n\right) $-equivariant
operator, where $H$ is the mean curvature vector field of the $G$-orbits in $%
F_{O}$.

From $D_{N\mathcal{F}}$ we now construct equivariant differential operators
on $M$ and $F_{O}\diagup G$, as follows. We define the operators%
\begin{equation*}
D_{M}^{\sigma }:=i_{\sigma }^{-1}\circ D_{N\mathcal{F}}\circ i_{\sigma
}:\Gamma \left( M,\mathcal{E}^{\sigma }\right) \rightarrow \Gamma \left( M,%
\mathcal{E}^{\sigma }\right) ,
\end{equation*}%
and%
\begin{equation*}
D_{F_{O}\diagup G}^{\rho }:=j_{\rho }^{-1}\circ D_{N\mathcal{F}}\circ
j_{\rho }:\Gamma \left( F_{O}\diagup G,\mathcal{T}^{\rho }\right)
\rightarrow \Gamma \left( F_{O}\diagup G,\mathcal{T}^{\rho }\right) .
\end{equation*}%
For an irreducible representation $\alpha :G\rightarrow U\left( V_{\alpha
}\right) $, let 
\begin{equation*}
\left( D_{M}^{\sigma }\right) ^{\alpha }:\Gamma \left( M,\mathcal{E}^{\sigma
}\right) ^{\alpha }\rightarrow \Gamma \left( M,\mathcal{E}^{\sigma }\right)
^{\alpha }
\end{equation*}%
be the restriction of $D_{M}^{\sigma }$ to sections of $G$-representation
type $\left[ \alpha \right] $. Similarly, for an irreducible representation $%
\beta :G\rightarrow U\left( W_{\beta }\right) $, let 
\begin{equation*}
\left( D_{F_{O}\diagup G}^{\rho }\right) ^{\beta }:\Gamma \left(
F_{O}\diagup G,\mathcal{T}^{\rho }\right) ^{\beta }\rightarrow \Gamma \left(
F_{O}\diagup G,\mathcal{T}^{\rho }\right) ^{\beta }
\end{equation*}%
be the restriction of $D_{F_{O}\diagup G}^{\rho }$ to sections of $O\left(
n\right) $-representation type $\left[ \beta \right] $. The proposition
below follows from Theorem \ref{IsomorphismsOfSectionsTheorem}.

\begin{proposition}
\label{SpectraSame}The operator $D_{M}^{\sigma }$ is transversally elliptic
and $G$-equivariant, and $D_{F_{O}\diagup G}^{\rho }$ is elliptic and $%
O\left( n\right) $-equivariant, and the closures of these operators are
self-adjoint. The operators $\left( D_{M}^{\sigma }\right) ^{\rho }$ and $%
\left( D_{F_{O}\diagup G}^{\rho }\right) ^{\sigma }$ have identical discrete
spectrum, and the corresponding eigenspaces are conjugate via Hilbert space
isomorphisms.
\end{proposition}

Thus, questions about the transversally elliptic operator $D_{M}^{\sigma }$
can be reduced to questions about the elliptic operators $D_{F_{O}\diagup
G}^{\rho }$ for each irreducible $\rho :G\rightarrow U\left( V_{\rho
}\right) $.

In particular, we are interested in the equivariant index, which we will
explain in detail the next section. In the following theorem, $\mathrm{ind}%
^{G}\left( \cdot \right) $ denotes the virtual representation-valued index
as explained in \cite{A} and in Section \ref{equivIndexSection}; the result
is a formal difference of finite-dimensional representations if the input is
a symbol of an elliptic operator.

\begin{theorem}
Suppose that $F_{O}$ is $G$-transversally spin$^{c}$. Then for every
transversally elliptic symbol class $\left[ u\right] \in K_{cpt,G}\left(
T_{G}^{\ast }M\right) $, there exists an operator of type $D_{M}^{\mathbf{1}%
} $ such that $\mathrm{ind}^{G}\left( u\right) =\mathrm{ind}^{G}\left(
D_{M}^{\mathbf{1}}\right) $.\label{indexClassGivenByTransvDiracThm}
\end{theorem}

\begin{exercise}
(continuation of Exercise \ref{Z2actiononTorusExercise})

\begin{itemize}
\item Determine a Dirac operator on the trivial $\mathbb{C}^{2}$ bundle over
the three-dimensional $F_{O}$. (Hint: use the Dirac operator from $\mathbb{R}%
^{3}$.)

\item Find the induced operator $D_{T^{2}}^{\mathbf{1}}$ on $T^{2}$, where $%
\mathbf{1}$ denotes the trivial representation $\mathbf{1}:O\left( 2\right)
\rightarrow \left\{ \mathbf{1}\right\} \in U\left( 1\right) $. This means
the restriction of the Dirac operator of $F_{O}$ to sections that are
invariant under the $O\left( 2\right) $ action.

\item Identify all irreducible unitary representations of $O\left( 2\right) $%
. (Hint: they are all homomorphisms $\sigma :O\left( 2\right) \rightarrow
U\left( 1\right) $.)

\item Find $\ker D_{T^{2}}^{\mathbf{1}}$ and $\ker D_{T^{2}}^{\mathbf{1\ast }%
}$, and decompose these vector spaces as direct sums of irreducible unitary
representations of $O\left( 2\right) $.

\item For each irreducible unitary representation $\rho :\mathbb{Z}%
_{2}\rightarrow U\left( 1\right) $ of $\mathbb{Z}_{2}$, determine the
induced operator $D_{F_{O}\diagup \mathbb{Z}_{2}}^{\rho }$.
\end{itemize}
\end{exercise}

\begin{exercise}
(continuation of Exercise \ref{SO3Exercise})

\begin{itemize}
\item Starting with a transversal Dirac operator on the trivial $\mathbb{C}%
^{2}$ bundle over $SO\left( 3\right) $, find the induced operator $%
D_{S^{2}}^{\mathbf{1}}$ on $S^{2}$.

\item Identify all irreducible unitary representations of $S^{1}$.

\item Find $\ker D_{S^{2}}^{\mathbf{1}}$ and $\ker D_{S^{2}}^{\mathbf{1\ast }%
}$, and decompose these vector spaces as direct sums of irreducible unitary
representations of $S^{1}$.
\end{itemize}
\end{exercise}

\begin{exercise}
(continuation of Exercise \ref{harmonicExercise}) Suppose that a compact,
connected Lie group $G$ acts on a Riemannian manifold. Show that the $\ker
\left( d+\delta \right) $ is the same as $G$-invariant part $\ker \left(
d+\delta \right) ^{\mathbf{1}}$ of $\ker \left( d+\delta \right) $, and $%
\mathrm{\ker }\left( d+\delta \right) ^{\rho }=0$ for all other irreducible
representations $\rho :G\rightarrow U\left( V_{\rho }\right) $.
\end{exercise}

\begin{exercise}
Suppose that $G=M$ acts freely on itself. Construct a transversal Dirac
operator acting on a trivial spinor bundle on the orthonormal frame bundle.
Determine the operator $D_{G}^{\mathbf{1}}$ for $1$ being the trivial
representation of $O\left( n\right) $, and find $\ker D_{G}^{\mathbf{1}}$
and $\ker $ $D_{G}^{\mathbf{1}\ast }$.
\end{exercise}

\section{Transverse index theory for $G$-manifolds and Riemannian foliations}

The research content and some of the expository content in this section are
joint work with J. Br\"{u}ning and F. W. Kamber, from \cite{BrKRi1} and \cite%
{BKRi2}.

\subsection{Introduction: the equivariant index\label{equivIndexSection}}

Suppose that a compact Lie group $G$ acts by isometries on a compact,
connected Riemannian manifold $M$, and let $E=E^{+}\oplus E^{-}$ be a
graded, $G$-equivariant Hermitian vector bundle over $M$. We consider a
first order $G$-equivariant differential operator $D=D^{+}:$ $\Gamma \left(
M,E^{+}\right) \rightarrow \Gamma \left( M,E^{-}\right) $ which is
transversally elliptic (as explained at the beginning of Section \ref%
{Z2actiononTorusExercise}). Let $D^{-}$ be the formal adjoint of $D^{+}$.

The group $G$ acts on $\Gamma \left( M,E^{\pm }\right) $ by $\left(
gs\right) \left( x\right) =g\cdot s\left( g^{-1}x\right) $, and the
(possibly infinite-dimensional) subspaces $\ker \left( D\right) $ and $\ker
\left( D^{\ast }\right) $ are $G$-invariant subspaces. Let $\rho
:G\rightarrow U\left( V_{\rho }\right) $ be an irreducible unitary
representation of $G$, and let $\chi _{\rho }=\mathrm{tr}\left( \rho \right) 
$ denote its character. Let $\Gamma \left( M,E^{\pm }\right) ^{\rho }$ be
the subspace of sections that is the direct sum of the irreducible $G$%
-representation subspaces of $\Gamma \left( M,E^{\pm }\right) $ that are
unitarily equivalent to the $\rho $ representation. The operators 
\begin{equation*}
D^{\rho }:\Gamma \left( M,E^{+}\right) ^{\rho }\rightarrow \Gamma \left(
M,E^{-}\right) ^{\rho }
\end{equation*}%
can be extended to be Fredholm operators, so that each irreducible
representation of $G$ appears with finite multiplicity in $\ker D^{\pm }$.
Let $a_{\rho }^{\pm }\in \mathbb{Z}^{+}$ be the multiplicity of $\rho $ in $%
\ker \left( D^{\pm }\right) $.

The virtual representation-valued index of $D$ (see \cite{A}) is 
\begin{equation*}
\mathrm{ind}^{G}\left( D\right) :=\sum_{\rho }\left( a_{\rho }^{+}-a_{\rho
}^{-}\right) \left[ \rho \right] ,
\end{equation*}%
where $\left[ \rho \right] $ denotes the equivalence class of the
irreducible representation $\rho $. The index multiplicity is 
\begin{equation*}
\mathrm{ind}^{\rho }\left( D\right) :=a_{\rho }^{+}-a_{\rho }^{-}=\frac{1}{%
\dim V_{\rho }}\mathrm{ind}\left( \left. D\right\vert _{\Gamma \left(
M,E^{+}\right) ^{\rho }\rightarrow \Gamma \left( M,E^{-}\right) ^{\rho
}}\right) .
\end{equation*}%
In particular, if $\mathbf{1}$ is the trivial representation of $G$, then 
\begin{equation*}
\mathrm{ind}^{\mathbf{1}}\left( D\right) =\mathrm{ind}\left( \left.
D\right\vert _{\Gamma \left( M,E^{+}\right) ^{G}\rightarrow \Gamma \left(
M,E^{-}\right) ^{G}}\right) ,
\end{equation*}%
where the superscript $G$ implies restriction to $G$-invariant sections.

There is a relationship between the index multiplicities and Atiyah's
equivariant distribution-valued index $\mathrm{ind}_{g}\left( D\right) $
(see \cite{A}); the multiplicities determine the distributional index, and
vice versa. Because the operator $\left. D\right\vert _{\Gamma \left(
M,E^{+}\right) ^{\rho }\rightarrow \Gamma \left( M,E^{-}\right) ^{\rho }}$
is Fredholm, all of the indices $\mathrm{ind}^{G}\left( D\right) $ , $%
\mathrm{ind}_{g}\left( D\right) $, and $\mathrm{ind}^{\rho }\left( D\right) $
depend only on the homotopy class of the principal transverse symbol of $D$.

The new equivariant index result (in \cite{BrKRi1})is stated in Theorem \ref%
{MainTheorem}. A\ large body of work over the last twenty years has yielded
theorems that express $\mathrm{ind}_{g}\left( D\right) $ in terms of
topological and geometric quantities (as in the Atiyah-Segal index theorem
for elliptic operators or the Berline-Vergne Theorem for transversally
elliptic operators --- see \cite{ASe},\cite{Be-V1},\cite{Be-V2}). However,
until now there has been very little known about the problem of expressing $%
\mathrm{ind}^{\rho }\left( D\right) $ in terms of topological or geometric
quantities which are determined at the different strata 
\begin{equation*}
M\left( \left[ H\right] \right) :=\bigcup\limits_{G_{x}\in \left[ H\right] }x
\end{equation*}%
of the $G$-manifold $M$. The special case when all of the isotropy groups
are finite was solved by M. Atiyah in \cite{A}, and this result was utilized
by T. Kawasaki to prove the Orbifold Index Theorem (see \cite{Kawas2}). Our
analysis is new in that the equivariant heat kernel related to the index is
integrated first over the group and second over the manifold, and thus the
invariants in our index theorem (Theorem \ref{MainTheorem}) are very
different from those seen in other equivariant index formulas. The explicit
nature of the formula is demonstrated in Theorem \ref%
{EulerCharacteristicTheorem}, a special case where the equivariant Euler
characteristic is computed in terms of invariants of the $G$-manifold strata.

One of the primary motivations for obtaining an explicit formula for $%
\mathrm{ind}^{\rho }\left( D\right) $ was to use it to produce a basic index
theorem for Riemannian foliations, thereby solving a problem that has been
open since the 1980s (it is mentioned, for example, in \cite{EK}). In fact,
the basic index theorem (Theorem \ref{BasicIndexTheorem}) is a consequence
of the equivariant index theorem. We note that a recent paper of Gorokhovsky
and Lott addresses this transverse index question on Riemannian foliations.
Using a different technique, they are able to prove a formula for the basic
index of a basic Dirac operator that is distinct from our formula, in the
case where all the infinitessimal holonomy groups of the foliation are
connected tori and if Molino's commuting sheaf is abelian and has trivial
holonomy (see \cite{GLott}). Our result requires at most mild topological
assumptions on the transverse structure of the strata of the Riemannian
foliation and has a similar form to the formula above for $\mathrm{ind}^{%
\mathbf{1}}\left( D\right) $. In particular, the analogue for the
Gauss-Bonnet Theorem for Riemannian foliations (Theorem \ref%
{BasicGaussBonnet}) is a corollary and requires no assumptions on the
structure of the Riemannian foliation.

There are several new techniques used in the proof of the equivariant index
theorem that have not been explored previously, and we will briefly describe
them in upcoming sections. First, the proof requires a modification of the
equivariant structure. In Section \ref{stratification}, we describe the
known structure of $G$-manifolds. In Section \ref{BlowupDoubleMainSection},
we describe a process of blowing up, cutting, and reassembling the $G$%
-manifold into what is called the \emph{desingularization}. The result is a $%
G$-manifold that has less intricate structure and for which the analysis is
more simple. We note that our desingularization process and the equivariant
index theorem were stated and announced in \cite{RiOber} and \cite{RiKor};
recently Albin and Melrose have taken it a step further in tracking the
effects of the desingularization on equivariant cohomology and equivariant
K-theory (\cite{AlMel}).

Another crucial step in the proof of the equivariant index theorem is the
decomposition of equivariant vector bundles over $G$-manifolds with one
orbit type. We construct a subbundle of an equivariant bundle over a $G$%
-invariant part of a stratum that is the minimal $G$-bundle decomposition
that consists of direct sums of isotypical components of the bundle. We call
this decomposition the \emph{fine decomposition} and define it in Section %
\ref{FineDecompSection}. More detailed accounts of this method are in \cite%
{BrKRi1}, \cite{KRi}.

\begin{exercise}
\vspace{1pt}\label{zeroMapExercise}Let $Z:C^{\infty }\left( S^{1},\mathbb{C}%
\right) \rightarrow \left\{ 0\right\} $ denote the zero operator on
complex-valued functions on the circle $S^{1}$. If we consider $Z$ to be an $%
S^{1}$-equivariant operator on the circle, find $\mathrm{ind}^{\rho }\left(
Z\right) $ for every irreducible representation $\rho :S^{1}\rightarrow
U\left( 1\right) $. (Important: the target bundle is the zero vector
bundle). We are assuming that $S^{1}$ acts by rotations.
\end{exercise}

\begin{exercise}
In Exercise \ref{zeroMapExercise}, instead calculate each $\mathrm{ind}%
^{\rho }\left( Z\right) $, if $Z:C^{\infty }\left( S^{1},\mathbb{C}\right)
\rightarrow C^{\infty }\left( S^{1},\mathbb{C}\right) $ is multiplication by
zero.
\end{exercise}

\begin{exercise}
Let $D=i\frac{d}{d\theta }:C^{\infty }\left( S^{1},\mathbb{C}\right)
\rightarrow C^{\infty }\left( S^{1},\mathbb{C}\right) $ be an operator on
complex-valued functions on the circle $S^{1}=\left\{ e^{i\theta }:\theta
\in \mathbb{R}\diagup 2\pi \mathbb{Z}\right\} $. Consider $D$ to be a $%
\mathbb{Z}_{2}$-equivariant operator, where the action is generated by $%
\theta \mapsto \theta +\pi $. Find $\mathrm{ind}^{\rho }\left( Z\right) $
for every irreducible representation $\rho :\mathbb{Z}_{2}\rightarrow
U\left( 1\right) $.
\end{exercise}

\begin{exercise}
Let $D=i\frac{d}{d\theta }:C^{\infty }\left( S^{1},\mathbb{C}\right)
\rightarrow C^{\infty }\left( S^{1},\mathbb{C}\right) $ be an operator on
complex-valued functions on the circle $S^{1}=\left\{ e^{i\theta }:\theta
\in \mathbb{R}\diagup 2\pi \mathbb{Z}\right\} $. Consider the $\mathbb{Z}%
_{2} $ action generated by $\theta \mapsto -\theta $. Show that $D$ is not $%
\mathbb{Z}_{2}$-equivariant.
\end{exercise}

\begin{exercise}
Let $\mathbb{Z}_{2}$ act on $T^{2}=$ $\mathbb{R}^{2}\diagup \mathbb{Z}^{2}$
with an action generated by $\left( x,y\right) =\left( -x,y\right) $ for $%
x,y\in \mathbb{R}\diagup \mathbb{Z}$. Calculate $\mathrm{ind}^{\rho }\left(
D\right) $ for every irreducible representation $\rho :\mathbb{Z}%
_{2}\rightarrow U\left( 1\right) $, where $D$ is the standard Dirac operator
on the trivial $\mathbb{C}^{2}$ bundle.
\end{exercise}

\subsection{Stratifications of $G$-manifolds\label{stratification}}

In the following, we will describe some standard results from the theory of
Lie group actions (see \cite{Bre}, \cite{Kaw}). As above, $G$ is a compact
Lie group acting on a smooth, connected, closed manifold $M$. We assume that
the action is effective, meaning that no $g\in G$ fixes all of $M$.
(Otherwise, replace $G$ with $G\diagup \left\{ g\in G:gx=x\text{ for all }%
x\in M\right\} $.) Choose a Riemannian metric for which $G$ acts by
isometries.

Given such an action and $x\in M$, the \textbf{isotropy or stabilizer
subgroup} $G_{x}<G$ is defined to be $\left\{ g\in G:gx=x\right\} $. The 
\textbf{orbit} $\mathcal{O}_{x}$ of a point $x$ is defined to be $\left\{
gx:g\in G\right\} $. Since $G_{xg}=gG_{x}g^{-1}$, the conjugacy class of the
isotropy subgroup of a point is fixed along an orbit.

On any such $G$-manifold, the conjugacy class of the isotropy subgroups
along an orbit is called the \textbf{orbit type}. On any such $G$-manifold,
there are a finite number of orbit types, and there is a partial order on
the set of orbit types. Given subgroups $H$ and $K$ of $G$, we say that $%
\left[ H\right] \leq $ $\left[ K\right] $ if $H$ is conjugate to a subgroup
of $K$, and we say $\left[ H\right] <$ $\left[ K\right] $ if $\left[ H\right]
\leq $ $\left[ K\right] $ and $\left[ H\right] \neq $ $\left[ K\right] $. We
may enumerate the conjugacy classes of isotropy subgroups as $\left[ G_{0}%
\right] ,...,\left[ G_{r}\right] $ such that $\left[ G_{i}\right] \leq \left[
G_{j}\right] $ implies that $i\leq j$. It is well-known that the union of
the principal orbits (those with type $\left[ G_{0}\right] $) form an open
dense subset $M_{0}$ of the manifold $M$, and the other orbits are called 
\textbf{singular}. As a consequence, every isotropy subgroup $H$ satisfies $%
\left[ G_{0}\right] \leq \left[ H\right] $. Let $M_{j}$ denote the set of
points of $M$ of orbit type $\left[ G_{j}\right] $ for each $j$; the set $%
M_{j}$ is called the \textbf{stratum} corresponding to $\left[ G_{j}\right] $%
. If $\left[ G_{j}\right] \leq \left[ G_{k}\right] $, it follows that the
closure of $M_{j}$ contains the closure of $M_{k}$. A stratum $M_{j}$ is
called a \textbf{minimal stratum} if there does not exist a stratum $M_{k}$
such that $\left[ G_{j}\right] <\left[ G_{k}\right] $ (equivalently, such
that $\overline{M_{k}}\subsetneq \overline{M_{j}}$). It is known that each
stratum is a $G$-invariant submanifold of $M$, and in fact a minimal stratum
is a closed (but not necessarily connected) submanifold. Also, for each $j$,
the submanifold $M_{\geq j}:=\bigcup\limits_{\left[ G_{k}\right] \geq \left[
G_{j}\right] }M_{k}$ is a closed, $G$-invariant submanifold.

Now, given a proper, $G$-invariant submanifold $S$ of $M$ and $\varepsilon
>0 $, let $T_{\varepsilon }(S)$ denote the union of the images of the
exponential map at $s$ for $s\in S$ restricted to the open ball of radius $%
\varepsilon $ in the normal bundle at $S$. It follows that $T_{\varepsilon
}(S)$ is also $G$ -invariant. If $M_{j}$ is a stratum and $\varepsilon $ is
sufficiently small, then all orbits in $T_{\varepsilon }\left( M_{j}\right)
\setminus M_{j}$ are of type $\left[ G_{k}\right] $, where $\left[ G_{k}%
\right] <\left[ G_{j}\right] $. This implies that if $j<k$, $\overline{M_{j}}%
\cap \overline{M_{k}}\neq \varnothing $, and $M_{k}\subsetneq \overline{M_{j}%
}$, then $\overline{M_{j}}$ and $\overline{M_{k}}$ intersect at right
angles, and their intersection consists of more singular strata (with
isotropy groups containing conjugates of both $G_{k}$ and $G_{j}$).

Fix $\varepsilon >0$. We now decompose $M$ as a disjoint union of sets $%
M_{0}^{\varepsilon },\dots ,M_{r}^{\varepsilon }$. If there is only one
isotropy type on $M$, then $r=0$, and we let $M_{0}^{\varepsilon }=\Sigma
_{0}^{\varepsilon }=M_{0}=M$. Otherwise, for $j=r,r-1,...,0$, let $%
\varepsilon _{j}=2^{j}\varepsilon $, and let%
\begin{equation*}
{\ }\Sigma _{j}^{\varepsilon }=M_{j}\setminus \overline{\bigcup_{k>j}M_{k}^{%
\varepsilon }};~~M_{j}^{\varepsilon }=T_{\varepsilon _{j}}\left(
M_{j}\right) \setminus \overline{\bigcup_{k>j}M_{k}^{\varepsilon }}.
\end{equation*}%
Thus, 
\begin{equation*}
{\ }T_{\varepsilon }\left( \Sigma _{j}^{\varepsilon }\right) \subset
M_{j}^{\varepsilon },~\Sigma _{j}^{\varepsilon }\subset M_{j}.
\end{equation*}

The following facts about this decomposition are contained in \cite[pp.~203ff%
]{Kaw}:

\begin{lemma}
\label{decomposition} For sufficiently small $\varepsilon >0$, we have, for
every $i\in \{0,\ldots ,r\}$:

\begin{enumerate}
\item $\displaystyle M=\coprod_{i=0}^{r}M_{i}^{\varepsilon }$ (disjoint
union).

\item $M_{i}^{\varepsilon }$ is a union of $G$-orbits; $\Sigma
_{i}^{\varepsilon }$ is a union of $G$-orbits.

\item The manifold $M_{i}^{\varepsilon }$ is diffeomorphic to the interior
of a compact $G$ -manifold with corners; the orbit space $M_{i}^{\varepsilon
}\diagup G$ is a smooth manifold that is isometric to the interior of a
triangulable, compact manifold with corners. The same is true for each $%
\Sigma _{i}^{\varepsilon }$.

\item If $\left[ G_{j}\right] $ is the isotropy type of an orbit in $%
M_{i}^{\varepsilon }$, then $j\leq i$ and $\left[ G_{j}\right] \leq \left[
G_{i}\right] $.

\item The distance between the submanifold $M_{j}$ and $M_{i}^{\varepsilon }$
for $j>i$ is at least $\varepsilon $.
\end{enumerate}
\end{lemma}

\begin{exercise}
Suppose $G$ and $M$ are as above. Show that if $\gamma $ is a geodesic that
is perpendicular at $x\in M$ to the orbit $\mathcal{O}_{x}$ through $x$,
then $\gamma $ is perpendicular to every orbit that intersects it.
\end{exercise}

\begin{exercise}
With $G$ and $M$ as above, suppose that $\gamma $ is a geodesic that is
orthogonal to a particular singular stratum $\Sigma $. Prove that each
element $g\in G$ maps $\gamma $ to another geodesic with the same property.
\end{exercise}

\begin{exercise}
Prove that if $S$ is any set of isometries of a Riemannian manifold $M$,
then the fixed point set $M^{S}:=\left\{ x\in M:gx=x\text{ for every }g\in
S\right\} $ is a totally geodesic submanifold of $M$.
\end{exercise}

\begin{exercise}
Prove that if $\Sigma $ is a stratum of the action of $G$ on $M$
corresponding to isotropy type $\left[ H\right] $, then the fixed point set $%
\Sigma ^{H}$ is a principal $N\left( H\right) \diagup H$ bundle over $%
G\diagdown \Sigma $, where $N\left( H\right) $ is the normalizer of the
subgroup $H$.
\end{exercise}

\begin{exercise}
\label{Z2xZ2actionOnS2exercise}Let $\mathbb{Z}_{2}\times \mathbb{Z}_{2}$ act
on $S^{2}\subset \mathbb{R}^{3}$ via $\left( x,y,z\right) \mapsto \left(
-x,y,z\right) $ and $\left( x,y,z\right) \mapsto \left( x,-y,z\right) $.
Determine the strata of this action and the isotropy types.
\end{exercise}

\begin{exercise}
\label{O2actionOnS2exercise}Let $O\left( 2\right) $ act on $S^{2}\subset 
\mathbb{R}^{3}$ by rotations that fix the $z$-axis. Determine the strata of
this action and the isotropy types.
\end{exercise}

\begin{exercise}
\label{S1actionOnKleinExercise}Let $M$ be the rectangle $\left[ 0,1\right]
\times \left[ -1,1\right] $ along with identifications $\left( s,1\right)
\sim \left( s,-1\right) $ for $0\leq s\leq 1$, $\left( 0,x\right) \sim
\left( 0,x+\frac{1}{2}\right) $ for $0\leq x\leq \frac{1}{2}$, and $\left(
1,x\right) \sim \left( 1,x+\frac{1}{2}\right) $ for $0\leq x\leq \frac{1}{2}$%
.

\begin{itemize}
\item Show that $M$ is a smooth Riemannian manifold when endowed with the
standard flat metric.

\item Find the topological type of the surface $M$.

\item Suppose that $S^{1}=\mathbb{R}\diagup 2\mathbb{Z}$ acts on $M$ via $%
\left( s,x\right) \mapsto \left( s,x+t\right) $, with $x,t\in \mathbb{R}%
\diagup 2\mathbb{Z}$ . Find the strata and the isotropy types of this action.
\end{itemize}
\end{exercise}

\subsection{Equivariant desingularization\label{BlowupDoubleMainSection}}

\vspace{1pt}Assume that $G$ is a compact Lie group that acts on a Riemannian
manifold $M$ by isometries. We construct a new $G$-manifold $N$ that has a
single stratum (of type $\left[ G_{0}\right] $) and that is a branched cover
of $M$, branched over the singular strata. A distinguished fundamental
domain of $M_{0}$ in $N$ is called the \textbf{desingularization} of $M$ and
is denoted $\widetilde{M}$. We also refer to \cite{AlMel} for their recent
related explanation of this desingularization (which they call \emph{%
resolution}).

A sequence of modifications is used to construct $N$ and $\widetilde{M}%
\subset N$. Let $M_{j}$ be a minimal stratum. Let $T_{\varepsilon }\left(
M_{j}\right) $ denote a tubular neighborhood of radius $\varepsilon $ around 
$M_{j}$, with $\varepsilon $ chosen sufficiently small so that all orbits in 
$T_{\varepsilon }\left( M_{j}\right) \setminus M_{j}$ are of type $\left[
G_{k}\right] $, where $\left[ G_{k}\right] <\left[ G_{j}\right] $. Let 
\begin{equation*}
N^{1}=\left( M\setminus T_{\varepsilon }\left( M_{j}\right) \right) \cup
_{\partial T_{\varepsilon }\left( M_{j}\right) }\left( M\setminus
T_{\varepsilon }\left( M_{j}\right) \right)
\end{equation*}%
be the manifold constructed by gluing two copies of $\left( M\setminus
T_{\varepsilon }\left( M_{j}\right) \right) $ smoothly along the boundary
(the codimension one case should be treated in a slightly different way; see 
\cite{BrKRi1} for details). Since the $T_{\varepsilon }\left( M_{j}\right) $
is saturated (a union of $G$-orbits), the $G$-action lifts to $N^{1}$. Note
that the strata of the $G$-action on $N^{1}$ correspond to strata in $%
M\setminus T_{\varepsilon }\left( M_{j}\right) $. If $M_{k}\cap \left(
M\setminus T_{\varepsilon }\left( M_{j}\right) \right) $ is nontrivial, then
the stratum corresponding to isotropy type $\left[ G_{k}\right] $ on $N^{1}$
is 
\begin{equation*}
N_{k}^{1}=\left( M_{k}\cap \left( M\setminus T_{\varepsilon }\left(
M_{j}\right) \right) \right) \cup _{\left( M_{k}\cap \partial T_{\varepsilon
}\left( M_{j}\right) \right) }\left( M_{k}\cap \left( M\setminus
T_{\varepsilon }\left( M_{j}\right) \right) \right) .
\end{equation*}%
Thus, $N^{1}$ is a $G$-manifold with one fewer stratum than $M$, and $%
M\setminus M_{j}$ is diffeomorphic to one copy of $\left( M\setminus
T_{\varepsilon }\left( M_{j}\right) \right) $, denoted $\widetilde{M}^{1}$
in $N^{1}$. In fact, $N^{1}$ is a branched double cover of $M$, branched
over $M_{j}$. If $N^{1}$ has one orbit type, then we set $N=N^{1}$ and $%
\widetilde{M}=\widetilde{M}^{1}$. If $N^{1}$ has more than one orbit type,
we repeat the process with the $G$-manifold $N^{1}$ to produce a new $G$%
-manifold $N^{2}$ with two fewer orbit types than $M$ and that is a $4$-fold
branched cover of $M$. Again, $\widetilde{M}^{2}$ is a fundamental domain of 
$\widetilde{M}^{1}\setminus \left\{ \text{a minimal stratum}\right\} $,
which is a fundamental domain of $M$ with two strata removed. We continue
until $N=N^{r}$ is a $G$-manifold with all orbits of type $\left[ G_{0}%
\right] $ and is a $2^{r}$-fold branched cover of $M$, branched over $%
M\setminus M_{0}$. We set $\widetilde{M}=\widetilde{M}^{r}$, which is a
fundamental domain of $M_{0}$ in $N$.

Further, one may independently desingularize $M_{\geq j}$, since this
submanifold is itself a closed $G$-manifold. If $M_{\geq j}$ has more than
one connected component, we may desingularize all components simultaneously.
The isotropy type of all points of $\widetilde{M_{\geq j}}$ is $\left[ G_{j}%
\right] $, and $\widetilde{M_{\geq j}}\diagup G$ is a smooth (open) manifold.

\begin{exercise}
Find the desingularization $\widetilde{M_{j}}$ of each stratum $M_{j}$ for
the $G$-manifold in Exercise \ref{Z2xZ2actionOnS2exercise}.
\end{exercise}

\begin{exercise}
Find the desingularization $\widetilde{M_{j}}$ of each stratum $M_{j}$ for
the $G$-manifold in Exercise \ref{O2actionOnS2exercise}.
\end{exercise}

\begin{exercise}
Find the desingularization $\widetilde{M_{j}}$ of each stratum $M_{j}$ for
the $G$-manifold in Exercise \ref{S1actionOnKleinExercise}.
\end{exercise}

\subsection{The fine decomposition of an equivariant bundle\label%
{FineDecompSection}}

\vspace{1pt}Let $X^{H}$ be the fixed point set of $H$ in a $G$-manifold $X$
with one orbit type $\left[ H\right] $. For $\alpha \in \pi _{0}\left(
X^{H}\right) $, let $X_{\alpha }^{H}$ denote the corresponding connected
component of $X^{H}$.

\begin{definition}
\label{componentRelGDefn}We denote $X_{\alpha }=GX_{\alpha }^{H}$, and $%
X_{\alpha }$ is called a \textbf{component of} $X$ \textbf{relative to} $G$.
\end{definition}

\begin{remark}
The space $X_{\alpha }$ is not necessarily connected, but it is the inverse
image of a connected component of $G\diagdown X$ under the projection $%
X\rightarrow G\diagdown X$. Also, note that $X_{\alpha }=X_{\beta }$ if
there exists $n$ in the normalizer $N=N\left( H\right) $ such that $%
nX_{\alpha }^{H}=X_{\beta }^{H}$. If $X$ is a closed manifold, then there
are a finite number of components of $X$ relative to $G$.
\end{remark}

We now introduce a decomposition of a $G$-bundle $E\rightarrow X$. Let $%
E_{\alpha }$ be the restriction $\left. E\right\vert _{X_{\alpha }^{H}}$.
For any irreducible representation $\sigma :H\rightarrow U\left( W_{\sigma
}\right) $, we define for $n\in N$ the representation $\sigma
^{n}:H\rightarrow U\left( W_{\sigma }\right) $ by $\sigma ^{n}\left(
h\right) =\sigma \left( n^{-1}hn\right) $. Let $\widetilde{N}_{\left[ \sigma %
\right] }=\left\{ n\in N:\left[ \sigma ^{n}\right] \text{~is~equivalent~to}~%
\left[ \sigma \right] ~\right\} $ . If the isotypical component $E_{\alpha
}^{\left[ \sigma \right] }$ is nontrivial, then it is invariant under the
subgroup $\widetilde{N}_{\alpha ,\left[ \sigma \right] }\subseteq \widetilde{%
N}_{\left[ \sigma \right] }$ that leaves in addition the connected component 
$X_{\alpha }^{H}$ invariant; again, this subgroup has finite index in $N$.
The isotypical components transform under $n\in N$ as%
\begin{equation*}
n:E_{\alpha }^{\left[ \sigma \right] }\overset{\cong }{\longrightarrow }E_{%
\overline{n}\left( \alpha \right) }^{\left[ \sigma ^{n}\right] }~,
\end{equation*}%
where $\overline{n}$ denotes the residue class class of $n\in N$ in $%
N\diagup \widetilde{N}_{\alpha ,\left[ \sigma \right] }~$. Then a
decomposition of $E$ is obtained by `inducing up' the isotypical components $%
E_{\alpha }^{\left[ \sigma \right] }$ from $\widetilde{N}_{\alpha ,\left[
\sigma \right] }$ to $N$. That is, 
\begin{equation*}
E_{\alpha ,\left[ \sigma \right] }^{N}=N\times _{\widetilde{N}_{\alpha ,%
\left[ \sigma \right] }}E_{\alpha }^{\left[ \sigma \right] }
\end{equation*}%
is a bundle containing $\left. E_{\alpha }^{\left[ \sigma \right]
}\right\vert _{X_{\alpha }^{H}}$ . This is an $N$-bundle over $NX_{\alpha
}^{H}\subseteq X^{H}$, and a similar bundle may be formed over each distinct 
$NX_{\beta }^{H}$, with $\beta \in \pi _{0}\left( X^{H}\right) $. Further,
observe that since each bundle $E_{\alpha ,\left[ \sigma \right] }^{N}$ is
an $N$-bundle over $NX_{\alpha }^{H}$, it defines a unique $G$ bundle $%
E_{\alpha ,\left[ \sigma \right] }^{G}$ (see Exercise \ref%
{NbundleDeterminesGbundleExercise}).

\begin{definition}
\label{fineComponentDefinition}The $G$-bundle $E_{\alpha ,\left[ \sigma %
\right] }^{G}$ over the submanifold $X_{\alpha }$ is called a \textbf{fine
component} or the \textbf{fine component of }$E\rightarrow X$ \textbf{%
associated to }$\left( \alpha ,\left[ \sigma \right] \right) $.
\end{definition}

If $G\diagdown X$ is not connected, one must construct the fine components
separately over each $X_{\alpha }$. If $E$ has finite rank, then $E$ may be
decomposed as a direct sum of distinct fine components over each $X_{\alpha
} $. In any case, $E_{\alpha ,\left[ \sigma \right] }^{N}$ is a finite
direct sum of isotypical components over each $X_{\alpha }^{H}$.

\begin{definition}
\label{FineDecompositionDefinition}The direct sum decomposition of $\left.
E\right\vert _{X_{\alpha }}$ into subbundles $E^{b}$ that are fine
components $E_{\alpha ,\left[ \sigma \right] }^{G}$ for some $\left[ \sigma %
\right] $, written 
\begin{equation*}
\left. E\right\vert _{X_{\alpha }}=\bigoplus\limits_{b}E^{b}~,
\end{equation*}%
is called the \textbf{refined isotypical decomposition} (or \textbf{fine
decomposition}) of $\left. E\right\vert _{X_{\alpha }}$.
\end{definition}

We comment that if $\left[ \sigma ,W_{\sigma }\right] $ is an irreducible $H$%
-representation present in $E_{x}$ with $x\in X_{\alpha }^{H}$, then $E_{x}^{%
\left[ \sigma \right] }$ is a subspace of a distinct $E_{x}^{b}$ for some $b$%
. The subspace $E_{x}^{b}$ also contains $E_{x}^{\left[ \sigma ^{n}\right] }$
for every $n$ such that $nX_{\alpha }^{H}=X_{\alpha }^{H}$~.

\begin{remark}
\label{constantMultiplicityRemark}Observe that by construction, for $x\in
X_{\alpha }^{H}$ the multiplicity and dimension of each $\left[ \sigma %
\right] $ present in a specific $E_{x}^{b}$ is independent of $\left[ \sigma %
\right] $. Thus, $E_{x}^{\left[ \sigma ^{n}\right] }$ and $E_{x}^{\left[
\sigma \right] }$ have the same multiplicity and dimension if $nX_{\alpha
}^{H}=X_{\alpha }^{H}$~.
\end{remark}

\begin{remark}
The advantage of this decomposition over the isotypical decomposition is
that each $E^{b}$ is a $G$-bundle defined over all of $X_{\alpha }$, and the
isotypical decomposition may only be defined over $X_{\alpha }^{H}$.
\end{remark}

\begin{definition}
\label{adaptedDefn}Now, let $E$ be a $G$-equivariant vector bundle over $X$,
and let $E^{b}~$be a fine component as in Definition \ref%
{fineComponentDefinition} corresponding to a specific component $X_{\alpha
}=GX_{\alpha }^{H}$ of $X$ relative to $G$. Suppose that another $G$-bundle $%
W$ over $X_{\alpha }$ has finite rank and has the property that the
equivalence classes of $G_{y}$-representations present in $E_{y}^{b},y\in
X_{\alpha }$ exactly coincide with the equivalence classes of $G_{y}$%
-representations present in $W_{y}$, and that $W$ has a single component in
the fine decomposition. Then we say that $W$ is \textbf{adapted} to $E^{b}$.
\end{definition}

\begin{lemma}
\label{AdaptedToAnyBundleLemma}In the definition above, if another $G$%
-bundle $W$ over $X_{\alpha }$ has finite rank and has the property that the
equivalence classes of $G_{y}$-representations present in $E_{y}^{b},y\in
X_{\alpha }$ exactly coincide with the equivalence classes of $G_{y}$%
-representations present in $W_{y}$, then it follows that $W$ has a single
component in the fine decomposition and hence is adapted to $E^{b}$. Thus,
the last phrase in the corresponding sentence in the above definition is
superfluous.
\end{lemma}

\begin{exercise}
\label{NbundleDeterminesGbundleExercise}Suppose that $X$ is a $G$-manifold, $%
H$ is an isotropy subgroup, and $E^{\prime }\rightarrow X^{H}$ is an $%
N\left( H\right) $-bundle over the fixed point set $X^{H}$. Prove that $%
E^{\prime }$ uniquely determines a $G$-bundle $E$ over $X$ such that $\left.
E\right\vert _{X^{H}}=E^{\prime }$.
\end{exercise}

\begin{exercise}
Prove Lemma \ref{AdaptedToAnyBundleLemma}.
\end{exercise}

\subsection{Canonical isotropy $G$-bundles}

In what follows, we show that there are naturally defined finite-dimensional
vector bundles that are adapted to any fine components. Once and for all, we
enumerate the irreducible representations $\left\{ \left[ \rho _{j},V_{\rho
_{j}}\right] \right\} _{j=1,2,...}$ of $G$. Let $\left[ \sigma ,W_{\sigma }%
\right] $ be any irreducible $H$-representation. Let $G\times _{H}W_{\sigma
} $ be the corresponding homogeneous vector bundle over the homogeneous
space $G\diagup H$. Then the $L^{2}$-sections of this vector bundle
decompose into irreducible $G$-representations. In particular, let $\left[
\rho _{j_{0}},V_{\rho _{j_{0}}}\right] $ be the equivalence class of
irreducible representations that is present in $L^{2}\left( G\diagup
H,G\times _{H}W_{\sigma }\right) $ and that has the lowest index $j_{0}$.
Then Frobenius reciprocity implies%
\begin{equation*}
0\neq \mathrm{Hom}_{G}\left( V_{\rho _{j_{0}}},L^{2}\left( G\diagup
H,G\times _{H}W_{\sigma }\right) \right) \cong \mathrm{Hom}_{H}\left( V_{%
\mathrm{\mathrm{Res}}\left( \rho _{j_{0}}\right) },W_{\sigma }\right) ,
\end{equation*}%
so that the restriction of $\rho _{j_{0}}$ to $H$ contains the $H$%
-representation $\left[ \sigma \right] $. Now, for a component $X_{\alpha
}^{H}$ of $X^{H}$, with $X_{\alpha }=GX_{\alpha }^{H}$ its component in $X$
relative to $G$, the trivial bundle%
\begin{equation*}
X_{\alpha }\times V_{\rho _{j_{0}}}
\end{equation*}%
is a $G$-bundle (with diagonal action) that contains a nontrivial fine
component $W_{\alpha ,\left[ \sigma \right] }$ containing $X_{\alpha
}^{H}\times \left( V_{\rho _{j_{0}}}\right) ^{\left[ \sigma \right] }$.

\begin{definition}
\label{canonicalIsotropyBundleDefinition}We call $W_{\alpha ,\left[ \sigma %
\right] }\rightarrow X_{\alpha }$ the \textbf{canonical isotropy }$G$\textbf{%
-bundle associated to }$\left( \alpha ,\left[ \sigma \right] \right) \in \pi
_{0}\left( X^{H}\right) \times \widehat{H}$. Observe that $W_{\alpha ,\left[
\sigma \right] }$ depends only on the enumeration of irreducible
representations of $G$, the irreducible $H$-representation $\left[ \sigma %
\right] $ and the component $X_{\alpha }^{H}$. We also denote the following
positive integers associated to $W_{\alpha ,\left[ \sigma \right] }$:

\begin{itemize}
\item $m_{\alpha ,\left[ \sigma \right] }=\dim \mathrm{Hom}_{H}\left(
W_{\sigma },W_{\alpha ,\left[ \sigma \right] ,x}\right) =\dim \mathrm{Hom}%
_{H}\left( W_{\sigma },V_{\rho _{j_{0}}}\right) $ (the \textbf{associated
multiplicity}), independent of the choice of $\left[ \sigma ,W_{\sigma }%
\right] $ present in $W_{\alpha ,\left[ \sigma \right] ,x}$ , $x\in
X_{\alpha }^{H}$ (see Remark \ref{constantMultiplicityRemark}).

\item $d_{\alpha ,\left[ \sigma \right] }=\dim W_{\sigma }$(the \textbf{%
associated representation dimension}), independent of the choice of $\left[
\sigma ,W_{\sigma }\right] $ present in $W_{\alpha ,\left[ \sigma \right]
,x} $ , $x\in X_{\alpha }^{H}$.

\item $n_{\alpha ,\left[ \sigma \right] }=\frac{\mathrm{rank}\left(
W_{\alpha ,\left[ \sigma \right] }\right) }{m_{\alpha ,\left[ \sigma \right]
}d_{\alpha ,\left[ \sigma \right] }}$ (the \textbf{inequivalence number}),
the number of inequivalent representations present in $W_{\alpha ,\left[
\sigma \right] ,x}$ , $x\in X_{\alpha }^{H}$.
\end{itemize}
\end{definition}

\begin{remark}
Observe that $W_{\alpha ,\left[ \sigma \right] }=W_{\alpha ^{\prime },\left[
\sigma ^{\prime }\right] }$ if $\left[ \sigma ^{\prime }\right] =\left[
\sigma ^{n}\right] $ for some $n\in N$ such that $nX_{\alpha }^{H}=X_{\alpha
^{\prime }}^{H}~$.
\end{remark}

\begin{lemma}
\label{canIsotropyGbundleAdaptedExists}Given any $G$-bundle $E\rightarrow X$
and any fine component $E^{b}$ of $E$ over some $X_{\alpha }=GX_{\alpha
}^{H} $, there exists a canonical isotropy $G$-bundle $W_{\alpha ,\left[
\sigma \right] }$ adapted to $E^{b}\rightarrow X_{\alpha }$.
\end{lemma}

\begin{exercise}
Prove Lemma \ref{canIsotropyGbundleAdaptedExists}.
\end{exercise}

\begin{exercise}
Suppose $G$ is a compact, connected Lie group, and $T$ is a maximal torus.
Let $G$ act on left on the homogeneous space $X=G\diagup T$.

\begin{itemize}
\item What is $\left( G\diagup T\right) ^{T}$?

\item Let $\sigma _{a}$ be a fixed irreducible representation of $T$ (on $%
\mathbb{C}$), say $\sigma _{a}\left( t\right) =\exp \left( 2\pi i\left(
a\cdot t\right) \right) $ with $a\in \mathbb{Z}^{m}$, $m=rank\left( T\right) 
$. Let $E=G\times _{\sigma _{a}}\mathbb{C}\rightarrow G\diagup T$ be the
associated line bundle. Is $E$ a canonical isotropy $G$-bundle associated to 
$\left( \cdot ,\left[ \sigma _{a}\right] \right) $?

\item Is it true that every complex $G$-bundle over $G\diagup T$ is a direct
sum of equivariant line bundles?
\end{itemize}
\end{exercise}

\subsection{The equivariant index theorem}

\vspace{1pt}To evaluate $\mathrm{ind}^{\rho }\left( D\right) $, we first
perform the equivariant desingularization as described in Section \ref%
{BlowupDoubleMainSection}, starting with a minimal stratum. In \cite{BrKRi1}%
, we precisely determine the effect of the desingularization on the
operators and bundles, and in turn the supertrace of the equivariant heat
kernel. We obtain the following result. In what follows, if $U$ denotes an
open subset of a stratum of the action of $G$ on $M$, $U^{\prime }$ denotes
the equivariant desingularization of $U$, and $\widetilde{U}$ denotes the
fundamental domain of $U$ inside $U^{\prime }$, as in Section \ref%
{BlowupDoubleMainSection}. We also refer the reader to Definitions \ref%
{componentRelGDefn} and \ref{canonicalIsotropyBundleDefinition}.

\begin{theorem}
(Equivariant Index Theorem, in \cite{BrKRi1}) \label{MainTheorem}Let $M_{0}$
be the principal stratum of the action of a compact Lie group $G$ on the
closed Riemannian $M$, and let $\Sigma _{\alpha _{1}}$,...,$\Sigma _{\alpha
_{r}}$ denote all the components of all singular strata relative to $G$. Let 
$E\rightarrow M$ be a Hermitian vector bundle on which $G$ acts by
isometries. Let $D:\Gamma \left( M,E^{+}\right) \rightarrow \Gamma \left(
M,E^{-}\right) $ be a first order, transversally elliptic, $G$-equivariant
differential operator. We assume that near each $\Sigma _{\alpha _{j}}$, $D$
is $G$-homotopic to the product $D_{N}\ast D^{\alpha _{j}}$, where $D_{N}$
is a $G$-equivariant, first order differential operator on $B_{\varepsilon
}\Sigma $ that is elliptic and has constant coefficients on the fibers and $%
D^{\alpha _{j}}\ $is a global transversally elliptic, $G$-equivariant, first
order operator on the $\Sigma _{\alpha _{j}}$. In polar coordinates 
\begin{equation*}
D_{N}=Z_{j}\left( \nabla _{\partial _{r}}^{E}+\frac{1}{r}D_{j}^{S}\right) ~,
\end{equation*}%
where $r$ is the distance from $\Sigma _{\alpha _{j}}$, where $Z_{j}$ is a
local bundle isometry (dependent on the spherical parameter), the map $%
D_{j}^{S}$ is a family of purely first order operators that differentiates
in directions tangent to the unit normal bundle of $\Sigma _{j}$. Then the
equivariant index $\mathrm{ind}^{\rho }\left( D\right) $ satisfies 
\begin{eqnarray*}
\mathrm{ind}^{\rho }\left( D\right) &=&\int_{G\diagdown \widetilde{M_{0}}%
}A_{0}^{\rho }\left( x\right) ~\widetilde{\left\vert dx\right\vert }%
~+\sum_{j=1}^{r}\beta \left( \Sigma _{\alpha _{j}}\right) ~, \\
\beta \left( \Sigma _{\alpha _{j}}\right) &=&\frac{1}{2\dim V_{\rho }}%
\sum_{b\in B}\frac{1}{n_{b}\mathrm{rank~}W^{b}}\left( -\eta \left(
D_{j}^{S+,b}\right) +h\left( D_{j}^{S+,b}\right) \right) \int_{G\diagdown 
\widetilde{\Sigma _{\alpha _{j}}}}A_{j,b}^{\rho }\left( x\right) ~\widetilde{%
\left\vert dx\right\vert }~,
\end{eqnarray*}%
where

\begin{enumerate}
\item $A_{0}^{\rho }\left( x\right) $ is the Atiyah-Singer integrand, the
local supertrace of the ordinary heat kernel associated to the elliptic
operator induced from $D^{\prime }$ (blown-up and doubled from $D$) on the
quotient $M_{0}^{\prime }\diagup G$, where the bundle $E$ is replaced by the
finite rank bundle $\mathcal{E}_{\rho }$ of sections of type $\rho $ over
the fibers.

\item Similarly, $A_{i,b}^{\rho }$ is the local supertrace of the ordinary
heat kernel associated to the elliptic operator induced from $\left( \mathbf{%
1}\otimes D^{\alpha _{j}}\right) ^{\prime }$ (blown-up and doubled from $%
\mathbf{1}\otimes D^{\alpha _{j}}$, the twist of $D^{\alpha _{j}}$ by the
canonical isotropy bundle $W^{b}\rightarrow \Sigma _{\alpha _{j}}$ ) on the
quotient $\Sigma _{\alpha _{j}}^{\prime }\diagup G$, where the bundle is
replaced by the space of sections of type $\rho $ over each orbit.

\item $\eta \left( D_{j}^{S+,b}\right) $ is the eta invariant of the
operator $D_{j}^{S+}$ induced on any unit normal sphere $S_{x}\Sigma
_{\alpha _{j}}$, restricted to sections of isotropy representation types in $%
W_{x}^{b}$; see \cite{BrKRi1}. This is constant on $\Sigma _{\alpha _{j}}$.

\item $h\left( D_{j}^{S+,b}\right) $ is the dimension of the kernel of $%
D_{j}^{S+,b}$, restricted to sections of isotropy representation types in $%
W_{x}^{b}$, again constant on on $\Sigma _{\alpha _{j}}$.

\item $n_{b}$ is the number of different inequivalent $G_{x}$-representation
types present in each $W_{x}^{b}$, $x\in \Sigma _{\alpha _{j}}$.
\end{enumerate}
\end{theorem}

\vspace{1pt}As an example, we immediately apply the result to the de Rham
operator and in doing so obtain an interesting equation involving the
equivariant Euler characteristic. In what follows, let $\mathcal{L}%
_{N_{j}}\rightarrow \Sigma _{j}$ be the orientation line bundle of the
normal bundle to the singular stratum $\Sigma _{j}$. The relative Euler
characteristic is defined for $X$ a closed subset of a manifold $Y$ as $\chi
\left( Y,X,\mathcal{V}\right) =\chi \left( Y,\mathcal{V}\right) -\chi \left(
X,\mathcal{V}\right) $, which is also the alternating sum of the dimensions
of the relative de Rham cohomology groups with coefficients in a complex
vector bundle $\mathcal{V}\rightarrow Y$. If $\mathcal{V}$ is an equivariant
vector bundle, the \textbf{equivariant Euler characteristic} $\chi ^{\rho
}\left( Y,\mathcal{V}\right) $ associated to the representation $\rho
:G\rightarrow U\left( V_{\rho }\right) $ is the alternating sum%
\begin{equation*}
\chi ^{\rho }\left( Y,\mathcal{V}\right) =\sum_{j}\left( -1\right) ^{j}\dim
H^{j}\left( Y,\mathcal{V}\right) ^{\rho },
\end{equation*}
where the superscript $\rho $ refers to the restriction of these cohomology
groups to forms of $G$-representation type $\left[ \rho \right] $. An
application of the equivariant index theorem yields the following result.

\begin{theorem}
(Equivariant Euler Characteristic Theorem, in \cite{BrKRi1}) \label%
{EulerCharacteristicTheorem}Let $M$ be a compact $G$-manifold, with $G$ a
compact Lie group and principal isotropy subgroup $H_{\mathrm{pr}}$. Let $%
M_{0}$ denote the principal stratum, and let $\Sigma _{\alpha _{1}}$,...,$%
\Sigma _{\alpha _{r}}$ denote all the components of all singular strata
relative to $G$. We use the notations for $\chi ^{\rho }\left( Y,X\right) $
and $\chi ^{\rho }\left( Y\right) $ as in the discussion above. Then 
\begin{eqnarray*}
\chi ^{\rho }\left( M\right) &=&\chi ^{\rho }\left( G\diagup H_{\mathrm{pr}%
}\right) \chi \left( G\diagdown M,G\diagdown \text{singular strata}\right) \\
&&+\sum_{j}\chi ^{\rho }\left( G\diagup G_{j}\text{~},\mathcal{L}%
_{N_{j}}\right) \chi \left( G\diagdown \overline{\Sigma _{\alpha _{j}}}%
,G\diagdown \text{lower strata}\right) ,
\end{eqnarray*}%
where $\mathcal{L}_{N_{j}}$ is the orientation line bundle of normal bundle
of the stratum component $\Sigma _{\alpha _{j}}$.
\end{theorem}

\begin{exercise}
Let $M=S^{n}$, let $G=O\left( n\right) $ acting on latitude spheres
(principal orbits, diffeomorphic to $S^{n-1}$). Show that there are two
strata, with the singular strata being the two poles. Show without using the
theorem by identifying the harmonic forms that%
\begin{equation*}
\chi ^{\rho }\left( S^{n}\right) =\left\{ 
\begin{array}{ll}
\left( -1\right) ^{n} & \text{if }\rho =\xi \\ 
1 & \text{if }\rho =\mathbf{1}%
\end{array}%
\right. ,
\end{equation*}%
where $\xi $ is the induced one dimensional representation of $O\left(
n\right) $ on the volume forms.
\end{exercise}

\begin{exercise}
In the previous example, show that 
\[
\chi ^{\rho }\left( G\diagup H_{\mathrm{%
pr}}\right) =\chi ^{\rho }\left( S^{n-1}\right) =\left\{ 
\begin{array}{ll}
\left( -1\right) ^{n-1} & \text{if }\rho =\xi \\ 
1 & \text{if }\rho =\mathbf{1}%
\end{array}%
\right. ,
\]
 and $\chi \left( G\diagdown M,G\diagdown \text{singular strata}%
\right) =-1$. Show that at each pole, 
\[
\chi ^{\rho }\left( G\diagup G_{j},%
\mathcal{L}_{N_{j}}\right) =\chi ^{\rho }\left( \mathrm{pt}\right) =\left\{ 
\begin{array}{ll}
1 & \text{if }\rho =\mathbf{1}\text{,} \\ 
0 & \text{otherwise.}%
\end{array}%
\right. ,
\]
 and $\chi \left( G\diagdown \overline{\Sigma _{\alpha _{j}}}%
,G\diagdown \text{lower strata}\right) =1$. Demonstrate that Theorem \ref%
{EulerCharacteristicTheorem} produces the same result as in the previous
exercise.
\end{exercise}

\begin{exercise}
If instead the group $\mathbb{Z}_{2}$ acts on $S^{n}$ by the antipodal map,
prove that%
\begin{equation*}
\chi ^{\rho }\left( S^{n}\right) =\left\{ 
\begin{array}{ll}
0 & \text{if }\rho =\mathbf{1}\text{ or }\xi \text{ and }n\text{ is odd} \\ 
1 & \text{if }\rho =\mathbf{1}\text{ or }\xi \text{ and }n\text{ is even} \\ 
0~ & \text{otherwise}%
\end{array}%
\right. ,
\end{equation*}%
both by direct calculation and by using Theorem \ref%
{EulerCharacteristicTheorem}.
\end{exercise}

\begin{exercise}
Consider the action of $\mathbb{Z}_{4}$ on the flat torus $T^{2}=\mathbb{R}%
^{2}\diagup \mathbb{Z}^{2}$, where the action is generated by a $\frac{\pi }{%
2}$ rotation. Explicitly, $k\in \mathbb{Z}_{4}$ acts on $\left( 
\begin{array}{c}
y_{1} \\ 
y_{2}%
\end{array}%
\right) $ by 
\begin{equation*}
\phi \left( k\right) \left( 
\begin{array}{c}
y_{1} \\ 
y_{2}%
\end{array}%
\right) =\left( 
\begin{array}{cc}
0 & -1 \\ 
1 & 0%
\end{array}%
\right) ^{k}\left( 
\begin{array}{c}
y_{1} \\ 
y_{2}%
\end{array}%
\right) .
\end{equation*}%
Endow $T^{2}$ with the standard flat metric. Let $\rho _{j}$ be the
irreducible character defined by $k\in \mathbb{Z}_{4}\mapsto e^{ikj\pi /2}$.
Prove that 
\begin{equation*}
\chi ^{\mathbf{1}}\left( T^{2}\right) =2,\chi ^{\rho _{1}}\left(
T^{2}\right) =\chi ^{\rho _{3}}\left( T^{2}\right) =-1,\chi ^{\rho
_{2}}\left( T^{2}\right) =0,
\end{equation*}%
in two different ways. First, compute the dimensions of the spaces of
harmonic forms to determine the equations. Second, use the Equivariant Euler
Characteristic Theorem.
\end{exercise}

\subsection{The basic index theorem for Riemannian foliations\label%
{BasicIndexSubsection}}

\vspace{1pt}Suppose that $E$ is a foliated Hermitian $\mathbb{C}\mathrm{l}%
\left( Q\right) $ module with metric basic $\mathbb{C}\mathrm{l}\left(
Q\right) $ connection $\nabla ^{E}$ over a Riemannian foliation $\left( M,%
\mathcal{F}\right) $. Let 
\begin{equation*}
D_{b}^{E}:\Gamma _{b}\left( E^{+}\right) \rightarrow \Gamma _{b}\left(
E^{-}\right)
\end{equation*}%
be the associated basic Dirac operator, as explained in Section \ref%
{basicDiracSection}.

In the formulas below, any lower order terms that preserve the basic
sections may be added without changing the index. Note that

\begin{definition}
The \emph{analytic basic index} of $D_{b}^{E}$ is 
\begin{equation*}
\mathrm{ind}_{b}\left( D_{b}^{E}\right) =\dim \ker D_{b}^{E}-\dim \ker
\left( D_{b}^{E}\right) ^{\ast }.
\end{equation*}
\end{definition}

As shown explicitly in \cite{BKRi2}, these dimensions are finite, and it is
possible to identify $\mathrm{ind}_{b}\left( D_{b}^{E}\right) $ with the
invariant index of a first order, $G$-equivariant differential operator $%
\widehat{D}$ over a vector bundle over a basic manifold $\widehat{W}$, where 
$G$ is $SO\left( q\right) $, $O\left( q\right) $, or the product of one of
these with a unitary group $U\left( k\right) $. By applying the equivariant
index theorem (Theorem \ref{MainTheorem}) to the case of the trivial
representation, we obtain the following formula for the index. In what
follows, if $U$ denotes an open subset of a stratum of $\left( M,\mathcal{F}%
\right) $, $U^{\prime }$ denotes the desingularization of $U$ very similar
to that in Section \ref{BlowupDoubleMainSection}, and $\widetilde{U}$
denotes the fundamental domain of $U$ inside $U^{\prime }$.

\begin{theorem}
(Basic Index Theorem for Riemannian foliations, in \cite{BKRi2}) \label%
{BasicIndexTheorem}Let $M_{0}$ be the principal stratum of the Riemannian
foliation $\left( M,\mathcal{F}\right) $, and let $M_{1}$, ... , $M_{r}$
denote all the components of all singular strata, corresponding to $O\left(
q\right) $-isotropy types $\left[ G_{1}\right] $, ... ,$\left[ G_{r}\right] $
on the basic manifold. With notation as in the discussion above, we have 
\begin{eqnarray*}
\mathrm{ind}_{b}\left( D_{b}^{E}\right) &=&\int_{\widetilde{M_{0}}\diagup 
\overline{\mathcal{F}}}A_{0,b}\left( x\right) ~\widetilde{\left\vert
dx\right\vert }+\sum_{j=1}^{r}\beta \left( M_{j}\right) ~ \\
\beta \left( M_{j}\right) &=&\frac{1}{2}\sum_{\tau }\frac{1}{n_{\tau }%
\mathrm{rank~}W^{\tau }}\left( -\eta \left( D_{j}^{S+,\tau }\right) +h\left(
D_{j}^{S+,\tau }\right) \right) \int_{\widetilde{M_{j}}\diagup \overline{%
\mathcal{F}}}A_{j,b}^{\tau }\left( x\right) ~\widetilde{\left\vert
dx\right\vert },
\end{eqnarray*}%
where the sum is over all components of singular strata and over all
canonical isotropy bundles $W^{\tau }$, only a finite number of which yield
nonzero $A_{j,b}^{\tau }$, and where

\begin{enumerate}
\item $A_{0,b}\left( x\right) $ is the Atiyah-Singer integrand, the local
supertrace of the ordinary heat kernel associated to the elliptic operator
induced from $\widetilde{D_{b}^{E}}$ (a desingularization of $D_{b}^{E}$) on
the quotient $\widetilde{M_{0}}\diagup \overline{\mathcal{F}}$, where the
bundle $E$ is replaced by the space of basic sections of over each leaf
closure;

\item $\eta \left( D_{j}^{S+,b}\right) $ and $h\left( D_{j}^{S+,b}\right) $
are defined in a similar way as in Theorem \ref{MainTheorem}, using a
decomposition $D_{b}^{E}=D_{N}\ast D_{M_{j}}$ at each singular stratum;

\item $A_{j,b}^{\tau }\left( x\right) $ is the local supertrace of the
ordinary heat kernel associated to the elliptic operator induced from $%
\left( \mathbf{1}\otimes D_{M_{j}}\right) ^{\prime }$ (blown-up and doubled
from $\mathbf{1}\otimes D_{M_{j}}$, the twist of $D_{M_{j}}$ by the
canonical isotropy bundle $W^{\tau }$) on the quotient $\widetilde{M_{j}}%
\diagup \overline{\mathcal{F}}$, where the bundle is replaced by the space
of basic sections over each leaf closure; and

\item $n_{\tau }$ is the number of different inequivalent $G_{j}$%
-representation types present in a typical fiber of $W^{\tau }$.
\end{enumerate}
\end{theorem}

An example of this result is the generalization of the Gauss-Bonnet Theorem
to the basic Euler characteristic. Recall from Section \ref%
{basicDiracSection} that the basic forms $\Omega \left( M,\mathcal{F}\right) 
$ are preserved by the exterior derivative, and the resulting cohomology is
called basic cohomology $H^{\ast }\left( M,\mathcal{F}\right) $. The basic
cohomology groups are finite-dimensional in the Riemannian foliation, and
the basic Euler characteristic is defined to be 
\begin{equation*}
\chi \left( M,\mathcal{F}\right) =\sum \left( -1\right) ^{j}\dim H^{j}\left(
M,\mathcal{F}\right) .
\end{equation*}

We have two independent proofs of the following Basic Gauss-Bonnet Theorem;
one proof uses the result in \cite{BePaRi}, and the other proof is a direct
consequence of the basic index theorem stated above (proved in \cite{BKRi2}%
). We express the basic Euler characteristic in terms of the ordinary Euler
characteristic, which in turn can be expressed in terms of an integral of
curvature. We extend the Euler characteristic notation $\chi \left( Y\right) 
$ for $Y$ any open (noncompact without boundary) or closed (compact without
boundary) manifold to mean%
\begin{equation*}
\chi \left( Y\right) =%
\begin{array}{ll}
\chi \left( Y\right) & \text{if }Y\text{ is closed} \\ 
\chi \left( 1\text{-point compactification of }Y\right) -1~ & \text{if }Y%
\text{ is open}%
\end{array}%
\end{equation*}%
Also, if $\mathcal{L}$ is a foliated line bundle over a Riemannian foliation 
$\left( X,\mathcal{F}\right) $, we define the basic Euler characteristic $%
\chi \left( X,\mathcal{F},\mathcal{L}\right) $ as before, using the basic
cohomology groups with coefficients in the line bundle $\mathcal{L}$.

\begin{theorem}
(Basic Gauss-Bonnet Theorem, announced in \cite{RiLodz}, proved in \cite%
{BKRi2}) \label{BasicGaussBonnet}Let $\left( M,\mathcal{F}\right) $ be a
Riemannian foliation. Let $M_{0}$,..., $M_{r}$ be the strata of the
Riemannian foliation $\left( M,\mathcal{F}\right) $, and let $\mathcal{O}%
_{M_{j}\diagup \overline{\mathcal{F}}}$ denote the orientation line bundle
of the normal bundle to $\overline{\mathcal{F}}$ in $M_{j}$. Let $L_{j}$
denote a representative leaf closure in $M_{j}$. With notation as above, the
basic Euler characteristic satisfies 
\begin{equation*}
\chi \left( M,\mathcal{F}\right) =\sum_{j}\chi \left( M_{j}\diagup \overline{%
\mathcal{F}}\right) \chi \left( L_{j},\mathcal{F},\mathcal{O}_{M_{j}\diagup 
\overline{\mathcal{F}}}\right) .
\end{equation*}
\end{theorem}

\begin{remark}
In \cite[Corollary 1]{GLott}, they show that in special cases the only term
that appears is one corresponding to a most singular stratum.
\end{remark}

\vspace{1pt}We now investigate some examples through exercises. The first
example is a codimension $2$ foliation on a 3-manifold. Here, $O(2)$ acts on
the basic manifold, which is homeomorphic to a sphere. In this case, the
principal orbits have isotropy type $\left( \{e\}\right) $, and the two
fixed points obviously have isotropy type $\left( O(2)\right) $. In this
example, the isotropy types correspond precisely to the infinitesimal
holonomy groups.

\begin{exercise}
\label{rotation} (From \cite{Ri1}, \cite{Ri2}, and \cite{BKRi2}) Consider
the one dimensional foliation obtained by suspending an irrational rotation
on the standard unit sphere $S^{2}$. \ On $S^{2}$ we use the cylindrical
coordinates $\left( z,\theta \right) $, related to the standard rectangular
coordinates by $x^{\prime }=\sqrt{\left( 1-z^{2}\right) }\cos \theta $, $%
y^{\prime }=\sqrt{\left( 1-z^{2}\right) }\sin \theta $, $z^{\prime }=z$. \
Let $\alpha $ be an irrational multiple of $2\pi $, and let the
three--manifold $M=S^{2}\times \left[ 0,1\right] /\sim $, where $\left(
z,\theta ,0\right) \sim \left( z,\theta +\alpha ,1\right) $. \ Endow $M$
with the product metric on $T_{z,\theta ,t}M\cong T_{z,\theta }S^{2}\times
T_{t}\mathbb{R}$. \ Let the foliation $\mathcal{F}$ be defined by the
immersed submanifolds $L_{z,\theta }=\cup _{n\in \mathbb{Z}}\left\{
z\right\} \times \left\{ \theta +\alpha \right\} \times \left[ 0,1\right] $
(not unique in $\theta $). \ The leaf closures $\overline{L}_{z}$ for $|z|<1$
are two dimensional, and the closures corresponding to the poles ($z=\pm 1$)
are one dimensional. Show that $\chi \left( M,\mathcal{F}\right) =2$, using
a direct calculation of the basic cohomology groups and also by using the
Basic Gauss-Bonnet Theorem.
\end{exercise}

The next example is a codimension $3$ Riemannian foliation for which all of
the infinitesimal holonomy groups are trivial; moreover, the leaves are all
simply connected. There are leaf closures of codimension 2 and codimension
1. The codimension 1 leaf closures correspond to isotropy type $(e)$ on the
basic manifold, and the codimension 2 leaf closures correspond to an
isotropy type $(O(2))$ on the basic manifold. In some sense, the isotropy
type measures the holonomy of the leaf closure in this case.

\begin{exercise}
(From \cite{BKRi2}) This foliation is a suspension of an irrational rotation
of $S^{1}$ composed with an irrational rotation of $S^{2}$ on the manifold $%
S^{1}\times S^{2}$. As in Example~\ref{rotation}, on $S^{2}$ we use the
cylindrical coordinates $\left( z,\theta \right) $, related to the standard
rectangular coordinates by $x^{\prime }=\sqrt{\left( 1-z^{2}\right) }\cos
\theta $, $y^{\prime }=\sqrt{\left( 1-z^{2}\right) }\sin \theta $, $%
z^{\prime }=z$. \ Let $\alpha $ be an irrational multiple of $2\pi $, and
let $\beta $ be any irrational number. We consider the four--manifold $%
M=S^{2}\times \left[ 0,1\right] \times \left[ 0,1\right] /\sim $, where $%
\left( z,\theta ,0,t\right) \sim \left( z,\theta ,1,t\right) $, $\left(
z,\theta ,s,0\right) \sim \left( z,\theta +\alpha ,s+\beta \mod1,1\right) $.
Endow $M$ with the product metric on $T_{z,\theta ,s,t}M\cong T_{z,\theta
}S^{2}\times T_{s}\mathbb{R}\times T_{t}\mathbb{R}$. Let the foliation $%
\mathcal{F}$ be defined by the immersed submanifolds $L_{z,\theta ,s}=\cup
_{n\in \mathbb{Z}}\left\{ z\right\} \times \left\{ \theta +\alpha \right\}
\times \left\{ s+\beta \right\} \times \left[ 0,1\right] $ (not unique in $%
\theta $ or $s$). The leaf closures $\overline{L}_{z}$ for $|z|<1$ are
three--dimensional, and the closures corresponding to the poles ($z=\pm 1$)
are two--dimensional. By computing the basic forms of all degrees, verify
that the basic Euler characteristic is zero. Next, use the Basic
Gauss-Bonnet Theorem to see the same result.
\end{exercise}

The following example is a codimension two transversally oriented Riemannian
foliation in which all the leaf closures have codimension one. The leaf
closure foliation is not transversally orientable, and the basic manifold is
a flat Klein bottle with an $O(2)$--action. The two leaf closures with $%
\mathbb{Z}_{2}$ holonomy correspond to the two orbits of type $\left( 
\mathbb{Z}_{2}\right) $, and the other orbits have trivial isotropy.

\begin{exercise}
This foliation is the suspension of an irrational rotation of the flat torus
and a $\mathbb{Z}_{2}$--action. Let $X$ be any closed Riemannian manifold
such that $\pi _{1}(X)=\mathbb{Z}\ast \mathbb{Z}$~, the free group on two
generators $\{\alpha ,\beta \}$. We normalize the volume of $X$ to be 1. Let 
$\widetilde{X}$ be the universal cover. We define $M=\widetilde{X}\times
S^{1}\times S^{1}\diagup \pi _{1}(X)$, where $\pi _{1}(X)$ acts by deck
transformations on $\widetilde{X}$ and by $\alpha \left( \theta ,\phi
\right) =\left( 2\pi -\theta ,2\pi -\phi \right) $ and $\beta \left( \theta
,\phi \right) =\left( \theta ,\phi +\sqrt{2}\pi \right) $ on $S^{1}\times
S^{1}$. We use the standard product--type metric. The leaves of $\mathcal{F}$
are defined to be sets of the form $\left\{ (x,\theta ,\phi )_{\sim
}\,|\,x\in \widetilde{X}\right\} $. Note that the foliation is transversally
oriented. Show that the basic Euler characteristic is $2$, in two different
ways.
\end{exercise}

The following example (from \cite{Car}) is a codimension two Riemannian
foliation that is not taut.

\begin{exercise}
For the example in Exercise \ref{CarExampleExercise}, show that the basic
manifold is a torus, and the isotropy groups are all trivial. Verify that $%
\chi \left( M,\mathcal{F}\right) =0$ in two different ways.
\end{exercise}

\end{document}